\documentclass[11pt, a4paper]{amsart}
\usepackage[utf8]{inputenc}

\usepackage{amsmath, amsfonts, amssymb, amsthm}

\usepackage{graphicx}
\usepackage[inline]{enumitem}
\usepackage{hyperref}

\usepackage{comment}
\usepackage{multicol}

\theoremstyle{definition}
\newtheorem{definition}{Definition}[subsection]

\theoremstyle{plain}
\newtheorem{proposition}[definition]{Proposition}
\newtheorem{propdef}[definition]{Proposition-Definition}
\newtheorem{lemma}[definition]{Lemma}
\newtheorem{theorem}[definition]{Theorem}
\newtheorem{corollary}[definition]{Corollary}

\usepackage{tikz, float, tikz-cd}
	\usetikzlibrary{arrows}

\numberwithin{equation}{section}
\numberwithin{figure}{section}

\newcommand{\N}{\mathbb{N}}

\newcommand{\Q}{\mathbb{Q}}

\newcommand{\C}{\mathbb{C}}
\newcommand{\K}{\mathbf{k}} 
\newcommand{\G}{\mathcal{G}} 
\newcommand{\V}{\mathcal{V}} 
\newcommand{\W}{\mathcal{W}} 
\newcommand{\M}{\mathcal{M}} 
\newcommand{\KX}{\mathbf{k}\langle\langle X \rangle \rangle}
\newcommand{\KY}{\mathbf{k}\langle\langle Y \rangle \rangle}
\newcommand{\KZ}{\mathbf{k}\langle\langle Z \rangle \rangle}
\newcommand{\QX}{\mathbb{Q}\langle\langle X \rangle \rangle}
\newcommand{\QY}{\mathbb{Q}\langle\langle Y \rangle \rangle}
\newcommand{\LibX}{\widehat{\mathfrak{Lib}}(X)}
\newcommand{\DMR}{\mathsf{DMR}} 
\newcommand{\dmr}{\mathfrak{dmr}} 
\newcommand{\Stab}{\mathsf{Stab}} 
\newcommand{\stab}{\mathfrak{stab}} 
\newcommand{\aut}{\mathrm{aut}}
\newcommand{\der}{\mathrm{der}}
\newcommand{\Aut}{\mathrm{Aut}}
\newcommand{\Mor}{\mathrm{Mor}}
\newcommand{\alg}{\mathrm{alg}}
\newcommand{\Mod}{\mathrm{mod}}
\newcommand{\End}{\mathrm{End}}

\newcommand{\Lib}{\mathfrak{Lib}}

\newcommand{\Ad}{\mathrm{Ad}}
\newcommand{\ad}{\mathrm{ad}}
\newcommand{\q}{\mathbf{q}}

\author{YADDADEN Khalef}
\address{Institut de Recherche Mathématique Avancée, UMR 7501, Université de Strasbourg, \, 7 rue René Descartes, 67000 Strasbourg, France}
\email{kyaddaden@math.unistra.fr}
\date{}
\thanks{} 
\title[Crossed product interpretation of the Double Shuffle Lie algebra]{Crossed product interpretation of the Double Shuffle Lie algebra attached to a finite Abelian group}
\begin{document}
	\maketitle
	\begin{abstract}
	    Racinet studied the scheme associated with the double shuffle and regularization relations between multiple polylogarithm values at $N^{th}$ roots of unity and constructed a group scheme attached to the situation; he also showed it to be the specialization for $G=\mu_N$ of a group scheme $\DMR_0^G$ attached to a finite abelian group $G$. Then, Enriquez and Furusho proved that $\DMR_0^G$ can be essentially identified with the stabilizer of a coproduct element arising in Racinet's theory with respect to the action of a group of automorphisms of a free Lie algebra attached to $G$. We reformulate Racinet's construction in terms of crossed products. Racinet's coproduct can then be identified with a coproduct $\hat{\Delta}^{\M}_G$ defined on a module $\hat{\M}_G$ over an algebra $\hat{\W}_G$, which is equipped with its own coproduct $\hat{\Delta}^{\W}_G$, and the group action on $\hat{\M}_G$ extends to a compatible action of $\hat{\W}_G$. We then show that the stabilizer of $\hat{\Delta}^{\M}_G$, hence $\DMR_0^G$, is contained in the stabilizer of $\hat{\Delta}^{\W}_G$. This yields an explicit group scheme containing $\DMR_0^G$, which we also express in the Racinet formalism. 
	\end{abstract}
	
	{\footnotesize \tableofcontents}
	
	\section*{Introduction}
A \emph{multiple L-value} (MLV in short) is a complex number defined by the following series
\begin{equation}
    L_{(k_1, \dots, k_r)}(z_1, \dots, z_r) := \sum_{0 < m_1 < \dots < m_r} \frac{z_1^{k_1} \cdots z_r^{k_r}}{m_1^{k_1} \cdots m_r^{k_r}}
	\label{MLV}
\end{equation}
where $r, k_1, \dots, k_r \in \N^{\ast}$ and $z_1, \dots, z_r$ in $\mu_N$ the group of $N$\textsuperscript{th} roots of unity in $\C$, where $N$ is an integer $\geq 1$. The series (\ref{MLV}) converges if and only if $(k_r, z_r) \neq (1, 1)$.
These values have been defined and studied by Goncharov in \cite{Gon98} and \cite{Gon01} and appear as a generalisation of the so called multiple zeta values which in turn generalise the special values of the Riemann zeta function.
Among the relations satisfied by the MLVs, our main interest here are the \emph{double shuffle and regularisation} ones. Understanding these relations has been greatly improved thanks to Racinet's work \cite{Rac}.

Essentially, he attached to each pair $(G, \iota)$ of a finite cyclic group $G$ and a group injection $\iota : G \to \C^{\times}$, a $\Q$-scheme $\DMR^{\iota}$ which associates to each commutative $\Q$-algebra $\K$, a set $\DMR^{\iota}(\K)$ that can be decomposed as a disjoint union of sets $\DMR^{\iota}_{\lambda}(\K)$ ($\lambda \in \K$). For any $\lambda \in \K$, $\DMR^{\iota}_{\lambda}(\K)$ is a subset of the algebra of non-commutative power series $\KX$ over formal non-commutative variables $x_0$ and $(x_g)_{g \in G}$ satisfying the following conditions : (a) group-likeness for the coproduct $\hat{\Delta} : \KX \to \KX^{\hat{\otimes} 2}$ for which the elements of $X$ are primitive (b) group-likeness of the image in $\KX/\KX x_0$ of a suitable correction of the element for the coproduct $\hat{\Delta}_{\star} : \KX/\KX x_0 \to \left(\KX/\KX x_0\right)^{\hat{\otimes} 2}$ (see \cite{Rac}, Definition 2.3.1) (c) conditions on the degree $1$ and $2$ terms of the element.
The double shuffle and regularisation relations on MLVs are then encoded in the statement that a suitable generating series of these values belongs to the set $\DMR^{\iota_{can}}_{i 2\pi}(\C)$ where $\iota_{can} : G=\mu_N \to \C^{\star}$ is the canonical embedding.
Racinet also proved that for any pair $(G, \iota)$, the set $\DMR^{\iota}_0(\K)$ equipped with the product $\circledast$ given in (\ref{circledast}) is a group that is independent of the choice of the embedding $\iota$, so we denote it $\DMR_0^G(\K)$. The pair $\G(\KX), \circledast)$ is a group (see Proposition-Definition \ref{group_law}) which contains $\DMR_0^G$ as a subgroup.
Thanks to \cite{Rac} Theorem I, the sets $\mathrm{DMR}^{\iota}_{\lambda}(\K)$ have a torsor structure over $(\DMR_0^G(\K), \circledast)$. This motivates the study of this group.

In order to improve the understanding of the group $(\DMR_0^G(\K), \circledast)$, Enriquez and Furusho related this group with the stabilizer $\Stab(\hat{\Delta}_{\star})(\K)$ of the coproduct $\hat{\Delta}_{\star}$ in \cite{EF0} for an action of $(\G(\KX), \circledast)$ on $\Mor_{\K-\Mod}(\KX/\KX x_0, (\KX/\KX x_0)^{\hat{\otimes} 2})$ arising from an action of the latter group on $\KX/\KX x_0$ (see §\ref{subsection_DMR0}).

In addition, Racinet's work also introduced a subalgebra $\KY$ of $\KX$ spanned by the words ending with $x_g$ for some $g \in G$. It is identified, as a $\K$-module, with $\KX/\KX x_0$ and is equipped with a coproduct $\KY \to \KY^{\hat{\otimes} 2}$ compatible with $\hat{\Delta}_{\star}$. For this reason, the former coproduct has also the same notation in \cite{Rac}.
However, we will adopt distinct notation for these two coproducts, by denoting respectively $\hat{\Delta}^{\alg}_{\star}$ and $\hat{\Delta}^{\Mod}_{\star}$ the coproducts on $\KY$ and $\KX/\KX x_0$.
The situation, detailed in §\ref{Racinet_formalism}, may be summarised by the diagram
\begin{equation}
    \label{sequence1}
    \begin{tikzcd}
        \mathbf{k}\langle\langle Y \rangle\rangle \arrow[rr, hook] &  & \mathbf{k}\langle\langle X \rangle\rangle \hspace{-0,5cm} & {} \arrow[loop, distance=2em, in=145, out=215] & \hspace{-1,10cm} \mathbf{k}\langle\langle X \rangle\rangle \arrow[rr, two heads] &  & \mathbf{k}\langle\langle X \rangle\rangle/\mathbf{k}\langle\langle X \rangle\rangle x_0
    \end{tikzcd}
\end{equation}
where the first arrow is an algebra morphism, the second one is the module structure of the algebra $\KX$ on itself, and the last one is a module morphism. The three last terms of sequence (\ref{sequence1}) are equipped with compatible actions of the group $(\G(\KX), \circledast)$ while the first and last terms are equipped with the compatible coproducts $\hat{\Delta}^{\alg}_{\star}$ and $\hat{\Delta}^{\Mod}_{\star}$. The stabilizer group construction of \cite{EF0} is then based on the fourth term of (\ref{sequence1}).

When $G=\{1\}$, it was proved in \cite{EF1} (Part 2, §3) that the subalgebra $\KY$ of $\KX$ is stable under the action of $(\G(\KX), \circledast)$ on $\KX$. One can therefore construct the stabilizer group $\Stab(\hat{\Delta}^{\alg}_{\star})(\K)$ of $\hat{\Delta}^{\alg}_{\star}$ with respect to the action of $(\G(\KX), \circledast)$ on $\Mor_{\K-\Mod}(\KY, \KY^{\hat{\otimes} 2})$. By \cite{EF2} (§3.1), one then has the inclusion $\Stab(\hat{\Delta}^{\Mod}_{\star})(\K) \subset \Stab(\hat{\Delta}^{\alg}_{\star})(\K)$.      

However, if $G \neq \{1\}$ one can see that the previous group action on $\KX$ no longer restricts to an action on $\KY$. This forbids a direct generalisation of the result of \cite{EF2}. Such a generalisation is obtained in §\ref{crossed_product} by introducing an algebra containing $\KX$, namely, the crossed product algebra $\KX \rtimes G$ (see Definition \ref{Def_Crossed_Product}) and developing a formalism on it parallel to Racinet's.
In this framework, there is a subalgebra $\hat{\W}_G$ of $\hat{\V}_G$ isomorphic to the algebra $\KY$ (see Proposition \ref{WG_free_algebra}) and a quotient module $\hat{\M}_G$ of the left regular $\hat{\V}_G$-module isomorphic to the module $\KX/\KX x_0$ (see Proposition \ref{isoMG}). The algebra $\hat{\W}_G$ is equipped with a bialgebra coproduct $\hat{\Delta}^{\W}_G$ and the module $\hat{\M}_G$ is equipped with a compatible coalgebra coproduct $\hat{\Delta}^{\M}_G$.
The group $(\G(\KX), \circledast)$ acts compatibly on the algebra $\hat{\V}_G$ and on its regular left module. In contrast to the situation with $\KY \subset \KX$, the action on the algebra $\hat{\V}_G$ restricts to the subalgebra $\hat{\W}_G$, while the action on the left regular $\hat{\V}_G$-module induces an action of the quotient module $\hat{\M}_G$.
This can be summarised in the following diagram
\begin{equation}
    \label{sequence2}
    \begin{tikzcd}
        \hat{\mathcal{W}}_G \arrow[rr, hook] &  & \hat{\mathcal{V}}_G \hspace{-0,5cm} & \arrow[loop, distance=2em, in=145, out=215] & \hspace{-1cm} \hat{\mathcal{V}}_G \arrow[rr, two heads] &  & \hat{\mathcal{M}}_G
    \end{tikzcd}
\end{equation}
This situation allows us to define two stabilizers : one denoted $\Stab(\hat{\Delta}^{\M}_{G})(\K)$ that is identified with $\Stab(\hat{\Delta}^{\Mod}_{\star})(\K)$ and another one denoted $\Stab(\hat{\Delta}^{\W}_{G})(\K)$.
One shows that the latter group is a generalisation of the group with the same notation defined in \cite{EF2} for $G=\{1\}$. One also shows the inclusion (see Theorem \ref{Stab_Inclusion}, generalising \cite{EF2}, Theorem 3.1)
\[
    \Stab(\hat{\Delta}^{\M}_{G})(\K) \subset \Stab(\hat{\Delta}^{\W}_{G})(\K).
\]

In §\ref{StabRac}, we express the group $\Stab(\hat{\Delta}^{\W}_{G})(\K)$ in Racinet's formalism by working out the suitable isomorphisms (see Proposition \ref{explicit_autY}).

In §\ref{LA_Side}, we show that the group functors $\K \mapsto \Stab(\hat{\Delta}^{\M}_{G})(\K)$ and $\K \mapsto \Stab(\hat{\Delta}^{\W}_{G})(\K)$ are affine $\Q$-group subschemes of $\K \mapsto (\G(\KX), \circledast)$ and study their Lie algebras. We show that these are stabilizer Lie algebras corresponding to the Lie algebra actions which are the infinitesimal versions of the $\Q$-group scheme morphisms obtained from the previous actions of the group $(\G(\KX), \circledast)$.
	\subsubsection*{\textbf{Acknowledgements}} The author is grateful to Benjamin Enriquez for the helpful discussions, ideas and careful reading.
	\subsubsection*{\textbf{Notation}} Throughout this paper, $G$ is a finite abelian group whose product will be denoted multiplicatively. For a commutative $\Q$-algebra $\K$, a $\K$-algebra $A$, an element $x \in A$ and an $A$-module $M$ we consider:
	\begin{itemize}[leftmargin=*]
        \item $\ell_{x} : M \to M$ (resp. $r_{x} : M \to M$) to be the $\K$-module endomorphism defined by $m \mapsto xm$ (resp. $m \mapsto mx$) and if $x$ is invertible, then $\ell_x$ (resp. $r_x$) is an automorphism.
        \item $\ad_x : A \to A$ to be the $\K$-module endomorphism given by $\ad_x(a) = [x, a] = x a - a x$.
        \item $\Ad_x : A \to A$ to be the $\K$-algebra automorphism defined by $a \mapsto x a x^{-1}$ with $x \in A^{\times}$.
    \end{itemize}
	\section{Racinet's formalism of the double shuffle theory} \label{Racinet_formalism}
In this part, we recall from \cite{Rac} the basic formalism of the double shuffle theory, the main ingredients being presented in §\ref{basic_objects}. In §\ref{double_shuffle_group} and §\ref{double_shuffle_LA} we introduce the double shuffle group and the double shuffle Lie algebra respectively; and we recall from \cite{EF0} the stabilizer interpretation of both objects.

\subsection{Basic objects of Racinet's formalism} \label{basic_objects}
Let $\K$ be a commutative $\Q$-algebra. Let $\KX$ be the free noncommutative associative series algebra with unit over the alphabet $X=\{x_0\} \sqcup \{x_g | g \in G\}$. It is complete graded with $\deg(x_0) = \deg(x_g) = 1$ for $g \in G$.
This algebra is endowed with a Hopf algebra structure for the coproduct $\hat{\Delta} : \KX \to {\KX}^{\hat{\otimes} 2}$, which is the unique morphism of topological $\K$-algebras given by $\hat{\Delta}(x_g) = x_g \otimes 1 + 1 \otimes x_g$, for any $g \in G \sqcup \{0\}$ (\cite{Rac}, §2.2.3).
Let then $\G(\KX)$ be the set of grouplike elements of $\KX$ for the coproduct $\hat{\Delta}$. It is a group for the product of $\KX$.  

The group $G$ acts on the set $X$, the permutation $t_g$ corresponding to $g \in G$ being given by $t_g(x_0) = x_0, \, t_g(x_h) = x_{gh}$ for $h \in G$. This action extends to an action by $\K$-algebra automorphisms on $\KX$ (\cite{Rac}, §3.1.1) which will also be denoted $g \mapsto t_g$. One can verify by checking on generators the identity:
\begin{equation}
    \forall g \in G, \, \hat{\Delta} \circ t_g = t_g^{\otimes 2} \circ \hat{\Delta},
    \label{Delta_circ_t}
\end{equation}
since both sides are given as a composition of $\K$-algebra morphisms. As a consequence of (\ref{Delta_circ_t}), for any $g \in G$, the $\K$-algebra automorphism $t_g : \KX \to \KX$ restricts to a group automorphism $t_g : \G(\KX) \to \G(\KX)$.

Throughout the document, let us denote $\KX \to \K^{\{\text{words in } x_0, (x_g)_{g \in G}\}}, v \mapsto \big((v | w)\big)_{w}$ the map such that $v = \sum_{w} (v | w) w$.

Each word in $X$ can be uniquely written $\left(x_0^{n_1}x_{g_1}x_0^{n_2}x_{g_2} \cdots x_0^{n_r}x_{g_r}x_0^{n_{r+1}}\right)_{\substack{r, n_1, \dots, n_{r+1} \in \N \\ g_1, \dots, g_r \in G}}$. This family forms a topological $\K$-module basis of $\KX$.
Let $\q$ be the $\K$-module automorphism of $\KX$ given by (\cite{Rac}, §2.2.7)
\begin{align}
    \q(x_0^{n_1-1}x_{g_1}x_0^{n_2-1}x_{g_2} \cdots x_0^{n_r-1}x_{g_r} & x_0^{n_{r+1}-1}) = \\
    & x_0^{n_1-1}x_{g_1}x_0^{n_2-1}x_{g_2g_1^{-1}} \cdots x_0^{n_r-1}x_{g_rg_{r-1}^{-1}}x_0^{n_{r+1}-1} \notag
\end{align}

For $(n, g) \in \N^{\ast} \times G$, set $y_{n, g} := x_0^{n-1}x_g$. Let $Y := \{ y_{n, g} | (n, g) \in \N^{\ast} \times G \}$. We define $\KY$ to be the topological free $\K$-algebra over $Y$, where for every $(n, g) \in \N^{\ast} \times G$, the element $y_{n, g}$ is of degree $n$. One can show that $\KY$ is equal to the $\K$-subalgebra $\K \oplus \bigoplus_{g \in G} \KX x_g$ of $\KX$ (\cite{Rac}, §2.2.5 and \cite{EF0}, §2.2).

One denotes by $\q_Y$ the $\K$-module automorphism of $\KY$ given by (\cite{Rac}, §2.2.7.)
\begin{equation}
    \q_Y(y_{n_1, g_1} \cdots y_{n_r, g_r}) := y_{n_1, g_1} y_{n_2, g_2g_1^{-1}} \cdots y_{n_r, g_rg_{r-1}^{-1}} 
\end{equation}

Let $\hat{\Delta}^{\alg}_{\star} : \KY \to \left(\KY \right)^{\hat{\otimes} 2}$ be the unique topological $\K$-algebra morphism such that for any $(n, g) \in \N^{\ast} \times G$
\begin{equation}
    \label{harmonic_coproduct}
    \hat{\Delta}^{\alg}_{\star}(y_{n,g}) = y_{n,g} \otimes 1 + 1 \otimes y_{n,g} + \sum_{\substack{k=1 \\ h \in G}}^{n-1} y_{k,h} \otimes y_{n-k,hg^{-1}}.
\end{equation}
The map $\hat{\Delta}^{\alg}_{\star}$ is called the \emph{harmonic coproduct} (\cite{Rac}, §2.3.1) and endows $\KY$ with a bialgebra structure. Moreover, one can easily check that the action $t$ on $\KX$ restricts to an action on $\KY$ by $\K$-algebra automorphisms.

The topological $\K$-module quotient $\KX / \KX x_0$ is a left $\KY$-module free of rank $1$.
The topological $\K$-module morphism $\pi_Y : \KX \to \KX / \KX x_0$ is a surjective map and its restriction to $\KY$ is a bijective map. It follows that there is a topological $\K$-module morphism $\hat{\Delta}_{\star}^{\Mod} : \KX / \KX x_0 \to \left(\KX / \KX x_0 \right)^{\hat{\otimes} 2}$ uniquely defined by the condition that the diagram
\begin{equation}
    \label{harmonic_coproduct_M}
    \begin{tikzcd}
        \KY \ar["\hat{\Delta}_{\star}^{\alg}"]{r} \ar["\pi_Y"']{d} & \Big(\KY \Big)^{\hat{\otimes} 2} \ar["(\pi_Y)^{\otimes 2}"]{d} \\
        \KX / \KX x_0 \ar["\hat{\Delta}_{\star}^{\Mod}"']{r}& \left(\KX / \KX x_0\right)^{\hat{\otimes} 2}
    \end{tikzcd}
\end{equation}
commutes. This equips $\KX / \KX x_0$ with a cocommutative coassociative coalgebra structure.

The $\K$-module automorphism $\q$ of $\KX$ preserves the submodule $\KX x_0$ and, therefore, induces a $\K$-module automorphism of $\KX / \KX x_0$ denoted $\overline{\q}$, which is intertwined with the $\K$-module automorphism $\q_Y$ of $\KY$ via the identification $\KY \simeq \KX/\KX x_0$.

\subsection{The double shuffle group \texorpdfstring{$\DMR_0^G(\K)$}{DMR0G(K)}} \label{double_shuffle_group}
\subsubsection{The group $(\G(\KX), \circledast)$}
Let $\K$ be a commutative $\Q$-algebra. Recall that the set of grouplike elements of $\KX$ for the coproduct $\hat{\Delta}$ is
\[
    \G(\KX) = \{ \Psi \in \KX^{\times} \, | \, \hat{\Delta}(\Psi) = \Psi \otimes \Psi \}.
\]

For $\Psi \in \G(\KX)$, let $\aut_{\Psi}$ be the topological $\K$-algebra automorphism of $\KX$ given by (\cite{EF0}, §4.1.3 based on \cite{Rac}, §3.1.2)
\begin{equation}
    x_0 \mapsto x_0 \qquad \text{and for } g \in G, x_g \mapsto \Ad_{t_g(\Psi^{-1})}(x_g).
    \label{aut_Psi}
\end{equation}
Define $S_{\Psi}$ to be the topological $\K$-module automorphism of $\KX$ given by (\cite{EF0}, (5.15) based on \cite{Rac}, (3.1.2.1))
\begin{equation}
    \label{eq_S_Psi}
    S_{\Psi} := \ell_{\Psi} \circ \aut_{\Psi}.
\end{equation}

\begin{lemma}
    For $\Psi \in \G(\KX)$, the $\K$-algebra automorphism $\aut_{\Psi}$ is a bialgebra automorphism of $\left(\KX, \hat{\Delta}\right)$.
    \label{aut_bialg}
\end{lemma}
\begin{proof}
    Both $\aut_{\Psi}$ and $\hat{\Delta}$ are $\K$-algebra automorphisms. So, using Identity (\ref{Delta_circ_t}), one can check on generators that
    \begin{equation}
        \hat{\Delta} \circ \aut_{\Psi} = \left(\aut_{\Psi}\right)^{\otimes 2} \circ \hat{\Delta}, 
    \end{equation}
    which is the wanted result.
\end{proof}

\begin{propdef}[\cite{Rac}, Proposition 3.1.6]
    \label{group_law}
    The pair $(\G(\KX), \circledast)$ is a group, where for $\Psi, \Phi \in \G(\KX)$,
    \begin{equation}
        \label{circledast}
        \Psi \circledast \Phi := S_{\Psi}(\Phi).
    \end{equation}
\end{propdef}

A proof of this claim is already available in Racinet's paper, however, considering the way it has been stated (using categorical considerations), it might be hard to read. Thus, we find it useful to rewrite it here. In order to do so, we will need this result:

\begin{lemma}
    \label{group_morphs}
    For $\Psi, \Phi \in \G(\KX)$, we have
    \begin{align}
        &\aut_{\Psi \circledast \Phi} = \aut_{\Psi} \circ \aut_{\Phi}&
        \label{aut_group_morph} \\
        &S_{\Psi \circledast \Phi} = S_{\Psi} \circ S_{\Phi}&
        \label{S_group_morph}
    \end{align}
\end{lemma}
\noindent This, in turn, uses the following technical Lemma which can be easily obtained by checking this identity on generators
\begin{lemma}
    \label{commut_t_aut}
    For $\Psi \in \G(\KX)$ and $g \in G$, we have
    \(
        \aut_{\Psi} \circ t_g = t_g \circ \aut_{\Psi}.
    \)
\end{lemma}

\begin{proof}[Proof of Lemma \ref{group_morphs}]
    It is enough to prove the identity (\ref{aut_group_morph}) on generators. Since for $\Psi \in \G(\KX)$ we have $\aut_{\Psi}(x_0)=x_0$, Identity (\ref{aut_group_morph}) is immediately true for $x_0$. Then, for $g \in G$, we have
    \[\begin{aligned}
        &\aut_{\Psi} \circ \aut_{\Phi}(x_g) = \aut_{\Psi} \circ \Ad_{t_g(\Phi^{-1})}(x_g) = \Ad_{\aut_{\Psi}\left(t_g(\Phi^{-1})\right)} \circ \aut_{\Psi}(x_{g}) \\
        & = \Ad_{\aut_{\Psi}\left(t_g(\Phi^{-1})\right)} \circ \Ad_{t_g(\Psi^{-1})}(x_g) = \Ad_{t_g\left(\aut_{\Psi}(\Phi^{-1})\right)t_g(\Psi^{-1})}(x_g) \\
        & = \Ad_{t_g\left(\aut_{\Psi}(\Phi^{-1})\Psi^{-1}\right)}(x_g) = \Ad_{t_g\left((\Psi \circledast \Phi)^{-1}\right)}(x_g) = \aut_{\Psi \circledast \Phi}(x_g)  
    \end{aligned}\]
    where the fourth equality is obtained by applying Lemma \ref{commut_t_aut}.
    This concludes the proof of Identity (\ref{aut_group_morph}). Finally, by using the latter, we get
    \[\begin{aligned}
        S_{\Psi} \circ S_{\Phi} = & \ell_{\Psi} \circ \aut_{\Psi} \circ \ell_{\Phi} \circ \aut_{\Phi} = \ell_{\Psi} \circ \ell_{\aut_{\Psi}(\Phi)} \circ \aut_{\Psi} \circ \aut_{\Phi} \\
        =& \ell_{\Psi\aut_{\Psi}(\Phi)} \circ \aut_{\Psi} \circ \aut_{\Phi} = \ell_{\Psi \circledast \Phi} \circ \aut_{\Psi \circledast \Phi} = S_{\Psi \circledast \Phi},
    \end{aligned}\]
    thus, establishing Identity (\ref{S_group_morph}).
\end{proof}

\begin{proof}[Proof of Proposition-Definition \ref{group_law}]
    \noindent From Lemma \ref{aut_bialg}, we deduce that $\circledast$ has its image in $\G(\KX)$. Next, thanks to Identity \ref{S_group_morph} in Lemma \ref{group_morphs}, the product $\circledast$ is associative. Indeed, for $\Psi, \Phi$ and $\Lambda \in \G(\KX)$, we have
    \[
        (\Psi \circledast \Phi) \circledast \Lambda
        = S_{\Psi \circledast \Phi}(\Lambda) 
        = S_{\Psi}\left(S_{\Phi}(\Lambda)\right)
        = S_{\Psi}(\Phi \circledast \Lambda)
        = \Psi \circledast (\Phi \circledast \Lambda).
    \]
    Finally, the other group axioms being easy to check, this proves Proposition \ref{group_law}.
\end{proof}

\begin{corollary}
    \label{group_morph} \ 
    \begin{enumerate}[label=(\alph*), leftmargin=*]
        \item There is a group action of $(\G(\KX), \circledast)$ on $\KX$ by $\K$-algebra automorphisms
        \begin{equation}
            (\G(\KX), \circledast) \longrightarrow \Aut_{\K-\alg}(\KX), \, \Psi \longmapsto \aut_{\Psi}
        \end{equation}
        \item There is a group action of $(\G(\KX), \circledast)$ on $\KX$ by $\K$-module automorphisms
        \begin{equation}
            (\G(\KX), \circledast) \longrightarrow \Aut_{\K-\Mod}(\KX), \, \Psi \longmapsto S_{\Psi}
        \end{equation}
    \end{enumerate}
\end{corollary}

\begin{proof}
    This result is exactly Lemma \ref{group_morphs}.
\end{proof}

Next, we aim to give a group action of $(\G(\KX), \circledast)$ on the topological $\K$-module $\KX / \KX x_0$ which is compatible with its action $S$ on $\KX$. It is important to notice that this action is not given by compatibility using $\pi_Y$ but by the following: 
\begin{propdef}[\cite{EF0}, §5.4] 
    \label{S_Y}
    For $\Psi \in \G(\KX)$, there is a unique $\K$-module automorphism $S_{\Psi}^Y$ of $\KX / \KX x_0$ such that the following diagram
    \begin{equation}
        \begin{tikzcd}
            \KX \ar[rrr, "S_{\Psi}"] \ar[d, "\overline{\q} \circ \pi_Y"'] &&& \KX \ar[d, "\overline{\q} \circ \pi_Y"] \\
            \KX / \KX x_0 \ar[rrr, "S_{\Psi}^Y"'] &&& \KX / \KX x_0
        \end{tikzcd}
    \end{equation}
    commutes.
\end{propdef}

\begin{corollary}
    \label{group_morph_SY}
    There is a group action of $(\G(\KX), \circledast)$ on $\KX / \KX x_0$ by topological $\K$-module automorphisms
    \begin{equation}
        (\G(\KX), \circledast) \longrightarrow \Aut_{\K-\Mod}\left(\KX / \KX x_0 \right), \, \Psi \longmapsto S_{\Psi}^Y
    \end{equation}
\end{corollary}

\begin{proof}
    We have
    \[
        S_{\Psi}^Y \circ S_{\Phi}^Y \circ \overline{\q} \circ \pi_Y =  S_{\Psi}^Y \circ \overline{\q} \circ \pi_Y \circ S_{\Phi}
        = \overline{\q} \circ \pi_Y \circ S_{\Psi} \circ S_{\Phi} =  \overline{\q} \circ \pi_Y \circ S_{\Psi \circledast \Phi},
    \]
    and, by uniqueness of the $\K$-module automorphism $S_{\Psi \circledast \Phi}^Y$, we obtain
    \[
        S_{\Psi}^Y \circ S_{\Phi}^Y = S_{\Psi \circledast \Phi}^Y.
    \]
\end{proof}

Let $\Gamma : \KX \to \K[[x]]^{\times}, \Psi \mapsto \Gamma_{\Psi}$ the function given by (\cite{Rac}, (3.2.1.2))
\begin{equation}
    \Gamma_{\Psi}(x) := \exp\left( \sum_{n \geq 2} \frac{(-1)^{n-1}}{n} (\Psi | x_0^{n-1} x_1) x^n \right).
    \label{Gamma_function}
\end{equation}
It satisfies the following property:
\begin{lemma}
    For $\Psi, \Phi \in \G(\KX)$, we have
    \(
        \Gamma_{\Psi \circledast \Phi} = \Gamma_{\Psi} \Gamma_{\Phi}.
    \)
    \label{Gamma_aut}
\end{lemma}
\begin{proof}
    Lemma 4.12 in \cite{EF0} says that the map $(- | x_0^{n-1}x_1) : (\G(\KX), \circledast) \to (\K, +)$ is a group morphism, for any $n \in \N^{\ast}$. The result is then obtained by straightforward computations.  
\end{proof}

\noindent We then define the following topological $\K$-module automorphism of $\KX / \KX x_0$:
\begin{equation}
    ^{\Gamma}S_{\Psi}^Y := \ell_{\Gamma^{-1}_{\Psi}(x_1)} \circ S_{\Psi}^Y.
\end{equation}

\begin{corollary}
    \label{Gamma_SY}
    There is a group action of $(\G(\KX), \circledast)$ on $\KX / \KX x_0$ by topological $\K$-module automorphisms
    \begin{equation}
        (\G(\KX), \circledast) \longrightarrow  \Aut_{\K-\Mod}\left(\KX / \KX x_0\right), \,\, \Psi \longmapsto \,^{\Gamma}S_{\Psi}^Y
    \end{equation}
\end{corollary}
\begin{proof}
    Follows from Corollary \ref{group_morph_SY} and Lemma \ref{Gamma_aut}.
\end{proof}

\noindent The above automorphism is related to an automorphism introduced in \cite{EF0}.
\begin{proposition}
    \label{link_ef0_yad}
    For any $\Psi \in \G(\KX)$, the $\K$-module automorphism $^{\Gamma}S_{\Psi}^Y$ is equal to the $\K$-module automorphism $S_{\Theta(\Psi)}^Y$ where \(\Theta : (\G(\KX), \circledast) \to ((\KX)^{\times}, \circledast)\footnote{The product $\circledast$ extends to a product on $\KX^{\times}$. See \cite{EF0}, Lemma 4.1 and \cite{Rac}, §3.1.2.}\) is the group morphism given by (\cite{EF0}, Proposition 4.13)
    \begin{equation}
        \Theta(\Psi):= \Gamma_{\Psi}^{-1}(x_1) \Psi \exp(-(\Psi|x_0)x_0).
    \end{equation}
\end{proposition}
\begin{proof}
    Let $\Psi \in \G(\KX)$ and $v \in \KX$. First, we have
    \[
        S_{\Theta(\Psi)}(v) = \Theta(\Psi) \aut_{\Theta(\Psi)}(v) = \left(\Gamma_{\Psi}^{-1}(x_1) \Psi \exp(-(\Psi|x_0)x_0)\right) \aut_{\Theta(\Psi)}(v)
    \]
    Moreover, one can check on generators that
    \[
        \aut_{\Theta(\Psi)} = \mathrm{Ad}_{\exp((\Psi | x_0) x_0)} \circ \aut_{\Psi}.
    \]
    Therefore, one obtains
    \[
        S_{\Theta(\Psi)}(v)
        = \Gamma_{\Psi}^{-1}(x_1) \Psi \aut_{\Psi}(v) \exp(-(\Psi | x_0) x_0)
        = \Gamma_{\Psi}^{-1}(x_1) S_{\Psi}(v) \exp(-(\Psi | x_0) x_0)
    \]
    Consequently,
    \begin{align*}
        ^{\Gamma}S_{\Psi}^Y\big(\overline{\q} \circ \pi_Y(v)\big)
        = & \Gamma_{\Psi}^{-1}(x_1) S_{\Psi}^Y \Big(\overline{\q} \circ \pi_Y(v) \Big)
        = \Gamma_{\Psi}^{-1}(x_1) \Big(\overline{\q} \circ \pi_Y \left(S_{\Psi}(v) \right)\Big) \\
        = & \overline{\q} \circ \pi_Y \left(\Gamma_{\Psi}^{-1}(x_1) S_{\Psi}(v)\right)
        = \overline{\q} \circ \pi_Y \left(S_{\Theta(\Psi)}(v) \right)
    \end{align*}
    This establishes the identity $^{\Gamma}S_{\Psi}^Y = S_{\Theta(\Psi)}^Y$, thanks to Proposition-Definition \ref{S_Y}.
\end{proof}

\subsubsection{The group $(\DMR_0^G(\K), \circledast)$} \label{subsection_DMR0}
Let $\K$ be a commutative $\Q$-algebra. For $\Psi \in \G(\KX)$, set $\Psi_{\star} := \overline{\q} \circ \pi_Y \left(\Gamma^{-1}_{\Psi}(x_1) \Psi\right) \in \KX/\KX x_0$.
\begin{propdef}[\cite{Rac}, Definition 3.2.1 and Theorem I]
    \label{DMR}
    If $G$ is a cyclic group, we define\footnote{The notation $\DMR$ is for "Double Mélange et Régularisation" which is French for "Double Shuffle and Regularisation".} $\DMR_0^G(\K)$\footnote{In \cite{Rac}, Definition 3.2.1 gives sets $\DMR_{\lambda}^{\iota}(\K)$ where $\lambda \in \K$ and $\iota : G \to \C^{\ast}$ a group embedding (therefore $G$ is cyclic). If $|G| \in \{1, 2\}$, the embedding $\iota$ is unique; and if $|G| \geq 3$, for $\lambda=0$, condition (iv) does not depend of the choice of $\iota$. For this reason, the embedding $\iota$ does not appear in our notation.} to be the set of $\Psi \in \G(\KX)$ such that :
    \begin{enumerate}[label=\roman*.]
        \begin{multicols}{2}
        \item $(\Psi | x_0) = (\Psi | x_1) = 0$;
        \item $\hat{\Delta}_{\star}^{\Mod}(\Psi_{\star}) = \Psi_{\star} \otimes \Psi_{\star}$;
        \item If $|G| \in \{1, 2\}, (\Psi | x_0 x_1) = 0$;
        \item If $|G| \geq 3, \forall g \in G, \, \left(\Psi | x_g - x_{g^{-1}}\right) = 0$.
        \end{multicols}
    \end{enumerate}
    The pair $(\DMR_0^G(\K), \circledast)$ is a subgroup of $(\G(\KX), \circledast)$.
\end{propdef}

Thanks to Corollary \ref{Gamma_SY}, there is a group action of $(\G(\KX), \circledast)$ on the $\K$-module $\Mor_{\K-\Mod}\left(\KX / \KX x_0, \left(\KX / \KX x_0\right)^{\Hat{\otimes} 2}\right)$ via :
\begin{equation}
    \label{act_on_Delta*}
    \Psi \cdot D := \left(\left({^{\Gamma}S_{\Psi}^Y}\right)^{\otimes 2}\right) \circ D \circ (^{\Gamma}S_{\Psi}^Y)^{-1}.
\end{equation}
In particular, the stabilizer of $D = \hat{\Delta}_{\star}^{\Mod}$ is the subgroup (\cite{EF0}, §5.4)
\begin{equation}
    \mathsf{Stab}(\hat{\Delta}_{\star}^{\Mod})(\K) := \left\{ \Psi \in \G(\KX) \, | \, \left( {^{\Gamma}S_{\Psi}^Y} \right)^{\otimes 2} \circ \hat{\Delta}_{\star}^{\Mod} = \hat{\Delta}_{\star}^{\Mod} \circ {^{\Gamma}S_{\Psi}^Y} \right\}.
\end{equation}

\begin{proposition}[\cite{EF0}, Theorem 1.2]
    \label{DMRsubsetStab}
    If $G$ is a cylic group, we have
    \begin{equation}
        \DMR_0^G(\K) = \{ \Psi \in \mathsf{Stab}(\hat{\Delta}_{\star}^{\Mod})(\K) \, | \, (\Psi | x_0) = (\Psi | x_1) = 0 \}.
    \end{equation}
\end{proposition}
\noindent Since the condition $(\Psi | x_0) = (\Psi | x_1) = 0$ defines a subgroup of $(\G(\KX), \circledast)$, Theorem \ref{DMRsubsetStab} then identifies $\DMR_0^G(\K)$ with the intersection of two subgroups of $(\G(\KX), \circledast)$.

\subsubsection{An affine \texorpdfstring{$\Q$}{Q}-group scheme structure}
Recall that an affine $\Q$-group scheme is a functor $\mathsf{G}$ from the category of commutative $\Q$-algebras to the category of groups for which is representable by a Hopf $\Q$-algebra (see, for example, \cite{Wat}, §1.2).
\begin{proposition}
    \label{group_schemes}
    The following assignments are affine $\Q$-group schemes:
    \begin{enumerate}[label=(\alph*), leftmargin=*]
        \item $\K \mapsto (\mathcal{G}(\KX), \circledast)$;
        \item $\mathsf{DMR}_0^G : \K \mapsto (\mathsf{DMR}_0^G(\K), \circledast)$;
        \item $\mathsf{Stab}(\hat{\Delta}^{\mathrm{mod}}_{\star}) : \K \mapsto \mathsf{Stab}(\hat{\Delta}^{\mathrm{mod}}_{\star})(\K)$. 
    \end{enumerate}
\end{proposition}
\begin{proof}
    \begin{enumerate*}[label=(\alph*), leftmargin=*]
        \item See \cite{EF0}, Lemma 4.6;
        \item See \cite{Rac}, Theorem I;
        \item See \cite{EF0}, Lemma 5.1.
    \end{enumerate*}
\end{proof}

\noindent Therefore, Proposition \ref{DMRsubsetStab} provides an inclusion of affine $\Q$-group schemes
\begin{equation}
    \label{schemes_inclusion}
    \mathsf{DMR}_0^G \subset \mathsf{Stab}(\hat{\Delta}^{\mathrm{mod}}_{\star}) \subset \Bigg(\K \mapsto (\mathcal{G}(\KX), \circledast)\Bigg)
\end{equation}

\subsection{The double shuffle Lie algebra \texorpdfstring{$\dmr_0^G$}{dmr0G}} \label{double_shuffle_LA}
Recall from Theorem 12.2 in \cite{Wat} that there exists a functor $\mathbf{Lie}$ from the category of affine $\Q$-group schemes to the category of $\Q$-Lie algebras such that $\mathbf{Lie}(\mathsf{G}) = \ker\Big(\mathsf{G}\left(\Q[\epsilon]/(\epsilon^2)\right) \to \mathsf{G}(\Q)\Big)$.
In this section, we provide an explicit formulation of the Lie algebras obtained by applying the functor $\mathbf{Lie}$ to the inclusions (\ref{schemes_inclusion}).

\subsubsection{The Lie algebra \texorpdfstring{$\left(\LibX, \langle \cdot, \cdot \rangle\right)$}{LibX}}
Let $\LibX$ be the free complete graded $\Q$-Lie algebra over the alphabet $X$. One can identify the $\Q$-algebra $\QX$ with the enveloping algebra of $\LibX$ (\cite{Rac}, §2.2.3).
Therefore, $\LibX$ is identified with the Lie subalgebra of primitive elements in $\QX$ for the coproduct $\hat{\Delta}$. Namely,
\begin{equation}
    \LibX \simeq \{ \psi \in \QX \, | \, \hat{\Delta}(\psi) = \psi \otimes 1 + 1 \otimes \psi \}.
\end{equation}

For $\psi \in \LibX$, let $d_{\psi}$ be the derivation of $\QX$ given by (\cite{Rac}, §3.1.12.2)
\begin{equation}
    \label{dpsi}
    d_{\psi}(x_0) = 0, \text{ \quad and for } g \in G, \, d_{\psi}(x_g) = [x_g, t_g(\psi)],
\end{equation}
and let $s_{\psi}$ be the $\Q$-linear endomorphism of $\QX$ given by (\cite{Rac}, §3.1.12.1)
\begin{equation}
    \label{eq_s_psi}
    s_{\psi} := \ell_{\psi} + d_{\psi}.
\end{equation}
We then define a Lie algebra bracket on $\LibX$ as follows (\cite{Rac}, §3.1.10.2):
\begin{equation}
    \forall \psi_1, \psi_2 \in \LibX, \quad \langle \psi_1, \psi_2 \rangle := s_{\psi_1}(\psi_2) - s_{\psi_2}(\psi_1).
    \label{new_bracket}
\end{equation}

\subsubsection{The Lie algebra $\left(\dmr_0^G, \langle \cdot, \cdot \rangle \right)$}
Let us define $\gamma : \QX \to \Q[[x]]$; $\psi \mapsto \gamma_{\psi}$, where
\begin{equation}
    \label{gamma_function}
    \gamma_{\psi}(x) := \sum_{n \in \N^{\ast}} \frac{(-1)^{n+1}}{n}\left(\psi | x_0^{n-1}x_1\right) x^n,
\end{equation}
and for $\psi \in \QX$, set $\psi_{\star} := \overline{\q} \circ \pi_Y (-\gamma_{\psi}(x_1) + \psi) \in \QX / \QX x_0$.
\begin{propdef}[\cite{Rac}, Definitions 3.3.1, 3.3.8 and Proposition 4.A.i)]
The set $\dmr_0^G$ of elements $\psi \in \LibX$ such that
\begin{enumerate}[leftmargin=*, label=\roman*.]
    \begin{multicols}{2}
    \item $(\psi | x_0) = (\psi | x_1) = 0$;
    \item $\hat{\Delta}_{\star}^{\Mod}(\psi_{\star}) = \psi_{\star} \otimes 1 + 1 \otimes \psi_{\star}$; \end{multicols}
    \item $\left(\psi_{\star} | x_0^{n-1}x_{g}\right) = (-1)^{n-1} \left(\psi_{\star} | x_0^{n-1}x_{g^{-1}}\right)$ for $(n, g) \in \N^{\ast} \times G$;
\end{enumerate}
is a complete graded Lie subalgebra of $\left(\LibX, \langle \cdot, \cdot \rangle\right)$. 
\end{propdef}
\noindent \textbf{Remark.} According to \cite{Rac}, Propositions 3.3.3 and 3.3.7, it is enough to have (iii) in these cases:
\[
    \begin{cases}
        \text{for } (n, g) = (2, 1) & \text{if } |G| = 2 \\ 
        \text{for } n = 1 \text{ and any } g \in G & \text{if } |G| \geq 3 
    \end{cases}
\]
since this identity is always true for all the other cases.

\subsubsection{Relation of $\dmr_0^G$ with a stabilizer Lie algebra}
\begin{proposition}[\cite{Rac}, (3.1.9.2)]
    There exists a Lie algebra action of $(\LibX, \langle \cdot, \cdot \rangle)$ by $\Q$-linear endomorphisms on $\QX$ given by
    \begin{equation}
        (\LibX, \langle \cdot, \cdot \rangle) \longrightarrow \mathrm{End}_{\Q}(\QX), \quad \psi \longmapsto s_{\psi}.
    \end{equation}
\end{proposition}
\begin{propdef}[\cite{Rac}, §4.1.1 and \cite{EF0}, Lemma 2.2]
    \label{s_Y_psi}
    For $\psi \in \LibX$, there exists a unique $\Q$-linear endomorphism\footnote{Racinet defined this $\Q$-linear endomorphism on $\Q\langle\langle Y \rangle\rangle$. Even if we proceeded differently, we chose to keep the notation for consistency.} $s^Y_{\psi}$ of $\QX / \QX x_0$ such that the following diagram
    \[
        \begin{tikzcd}
            \QX \ar[r, "s_{\psi}"] \ar[d, "\overline{\q} \circ \pi_Y"']& \QX \ar[d, "\overline{\q} \circ \pi_Y"] \\
            \QX / \QX x_0 \ar[r, "s^Y_{\psi}"'] & \QX / \QX x_0
        \end{tikzcd}
    \]
    commutes. Moreover, there is a Lie algebra action of $(\LibX, \langle \cdot, \cdot \rangle)$ by $\Q$-linear endormorphisms on $\QX / \QX x_0$ given by
    \begin{equation}
        \left(\LibX, \langle \cdot, \cdot \rangle\right) \longrightarrow \mathrm{End}_{\Q}\left(\QX / \QX x_0\right), \quad \psi \longmapsto s^Y_{\psi}.
    \end{equation}
\end{propdef}

\noindent For $\psi \in \LibX$, we consider the following $\Q$-linear endomorphism on $\QX / \QX x_0$
\begin{equation}
    ^{\gamma}s^Y_{\psi} := \ell_{-\gamma_{\psi}(x_1)} + s^Y_{\psi}.
\end{equation}
\begin{lemma}
    \label{gamma_s_s_theta}
    For any $\psi \in \LibX$, the $\Q$-linear endomorphism $^{\gamma}s^Y_{\psi}$ is equal to $s_{\theta(\psi)}^Y$ where \(\theta : (\LibX, \langle \cdot, \cdot \rangle) \to (\QX, \langle \cdot, \cdot \rangle)\) is the Lie algebra morphism\footnote{One can equip $\QX$ with the bracket $\langle \cdot, \cdot \rangle$ as described in (\ref{new_bracket}).} given by (\cite{EF0}, Proposition 2.5)
    \begin{equation}
        \theta(\psi) := -\gamma_{\psi}(x_1) + \psi - (\psi | x_0) x_0
    \end{equation}
\end{lemma}
\begin{proof}
    Let $\psi \in \LibX$ and $a \in \QX$. First, we have
    \[
        s_{\theta(\psi)}(a) = \theta(\psi) a + d_{\theta(\psi)}(a) = (-\gamma_{\psi}(x_1)+\psi-(\psi | x_0) x_0) a + d_{\theta(\psi)}(a) 
    \]
    Moreover, one can check on generators that
    \[
        d_{\theta(\psi)} = \mathrm{ad}_{(\psi | x_0) x_0} + d_{\psi},
    \]
    Therefore, one obtains
    \begin{align*}
        s_{\theta(\psi)}(a) = & (-\gamma_{\psi}(x_1) + \psi-(\psi | x_0) x_0) a + \mathrm{ad}_{(\psi | x_0) x_0}(a) + d_{\psi}(a) \\
        = &-\gamma_{\psi}(x_1) a + s_{\psi}(a) - (\psi | x_0) a x_0.    
    \end{align*}
    Consequently, \newline
        \(\begin{aligned}
        ^{\gamma}s^Y_{\psi}(\overline{\q} \circ \pi_Y(a)) = & -\gamma_{\psi}(x_1) \, \Big(\overline{\q} \circ \pi_Y(a)\Big) + s^Y_{\psi}\Big(\overline{\q} \circ \pi_Y(a)\Big) \\ 
        = & \overline{\q} \circ \pi_Y \Big(-\gamma_{\psi}(x_1) a \Big) + \overline{\q} \circ \pi_Y \Big(s_{\psi}(a)\Big) \\
        = & \overline{\q} \circ \pi_Y \big(-\gamma_{\psi}(x_1) a + s_{\psi}(a)\big) = \overline{\q} \circ \pi_Y \Big( s_{\theta(\psi)}(a) \Big).  
    \end{aligned}\) \newline
    This establishes the identity $^{\gamma}s^Y_{\psi} = s^Y_{\theta(\psi)}$, thanks to Proposition-Definition \ref{s_Y_psi}. 
\end{proof}

\begin{proposition}
    \label{gamma_sY}
    There is a Lie algebra action of $(\LibX, \langle \cdot, \cdot \rangle)$ by $\Q$-linear endomorphisms on $\QX / \QX x_0$ by
    \begin{equation}
        \left(\LibX, \langle \cdot, \cdot \rangle\right) \longrightarrow \mathrm{End}_{\Q}\left(\QX / \QX x_0\right), \quad \psi \longmapsto \, ^{\gamma}s_{\psi}.
    \end{equation}
\end{proposition}
\begin{proof}
    Thanks to \cite{EF0}, §2.5, the map $\psi \mapsto s_{\theta(\psi)}$ is a Lie algebra action of $(\LibX, \langle \cdot, \cdot \rangle)$ on $\QX / \QX x_0$. The result then follows from Lemma \ref{gamma_s_s_theta}.
\end{proof}

The space $\mathrm{Mor}_{\Q}\left(\QX / \QX x_0, (\QX / \QX x_0)^{\Hat{\otimes} 2}\right)$ is then equipped with an action of the Lie algebra $(\LibX, \langle \cdot, \cdot\rangle)$ given by (\cite{EF0}, §2.5) 
\begin{equation}
    \label{LA_act_on_Delta*}
    \psi \cdot D := \left(^{\gamma}s^Y_{\psi} \otimes \mathrm{id} + \mathrm{id} \otimes \, ^{\gamma}s^Y_{\psi}\right) \circ D - D \circ \, ^{\gamma}s^Y_{\psi},   
\end{equation}
where $\psi \in \LibX$ and $D \in \mathrm{Mor}_{\Q}\left(\QX / \QX x_0, \left(\QX / \QX x_0\right)^{\Hat{\otimes} 2}\right)$.

The stabilizer Lie algebra $\mathfrak{stab}(\hat{\Delta}_{\star}^{\Mod})$ of $D = \hat{\Delta}_{\star}^{\Mod}$ is then the Lie subalgebra of $(\LibX, \langle \cdot, \cdot \rangle)$ given by (\cite{EF0}, §2.5)
\begin{equation}
    \label{stab_Delta*}
    \mathfrak{stab}(\hat{\Delta}_{\star}^{\Mod}) := \{ \psi \in \LibX \, | \, (^{\gamma}s^Y_{\psi} \otimes \mathrm{id} + \mathrm{id} \otimes \, ^{\gamma}s^Y_{\psi}) \circ \hat{\Delta}_{\star}^{\Mod} = \hat{\Delta}_{\star}^{\Mod} \circ \, ^{\gamma}s^Y_{\psi} \}.
\end{equation}

\noindent It is related to the Lie algebra $\dmr_0^G$ as follows:
\begin{proposition}
    \label{stab_subset_stab*}
    $\dmr_0^G \subset \mathfrak{stab}(\hat{\Delta}_{\star}^{\Mod})$ (as Lie subalgebras of $(\LibX, \langle \cdot, \cdot \rangle)$).
\end{proposition}
\begin{proof}
    Thanks to Lemma \ref{gamma_s_s_theta}, the stabilizer Lie algebra $\stab(\hat{\Delta}^{\Mod}_{\star})$ is identified with the stabilizer Lie algebra given in \cite{EF0}. Therefore the wanted inclusion is stated in Corollary 3.11 of \cite{EF0} ($\dmr_0^G$ being denoted $\dmr_0$ in \cite{EF0}).
\end{proof}

\subsubsection{Exponential maps} \label{exp_map}
Recall the affine $\Q$-group schemes from Proposition \ref{group_schemes}. We have :
\begin{proposition}
    \label{LA_of_Rac_groups}
    \begin{enumerate}[label=(\alph*), leftmargin=*]
        \item $\mathbf{Lie}\Big(\K \mapsto (\G(\KX), \circledast)\Big) = \left(\LibX, \langle \cdot, \cdot \rangle\right)$;
        \item $\mathbf{Lie}(\DMR_0^G, \circledast) = \left(\dmr_0^G, \langle \cdot, \cdot \rangle\right)$, where $G$ is a cyclic group;
        \item \label{LA_stab_Stab} $\mathbf{Lie}(\Stab(\hat{\Delta}^{\Mod}_{\star}), \circledast) = \left(\stab(\hat{\Delta}^{\Mod}_{\star}), \langle \cdot, \cdot \rangle\right)$. 
    \end{enumerate}
\end{proposition}
\begin{proof}
    \begin{enumerate*}[label=(\alph*), leftmargin=*]
        \item See \cite{EF0}, §4.1.4;
        \item See \cite{Rac}, §3.3.8;
        \item See \cite{EF0}, (5.12).
    \end{enumerate*}
\end{proof}

Let $\K$ be a commutative $\Q$-algebra. Let us denote $\widehat{\Lib}_{\K}(X) := \widehat{\Lib}(X) \hat{\otimes} \K$.
Let $\mathrm{cbh}_{\langle\cdot, \cdot\rangle} : \widehat{\Lib}_{\K}(X) \times \widehat{\Lib}_{\K}(X) \to \widehat{\Lib}_{\K}(X)$ be the map defined by $\mathrm{cbh}_{\langle\cdot, \cdot\rangle}(\psi, \phi) := \mathrm{mor}_{\psi, \phi}(\mathrm{cbh})$, where $\mathrm{cbh}$ in $\widehat{\Lib}_{\Q}(a,b)$ is the Campbell-Baker-Hausdorff series (\cite{EF0}, §4.1.2) $\mathrm{cbh} = \log(\exp(a)\exp(b))$ with $\log : 1 + \Q\langle\langle a, b \rangle\rangle \to \Q\langle\langle a, b \rangle\rangle_0$ and $\mathrm{mor}_{\psi, \phi}$ is the Lie algebra morphism $\widehat{\Lib}_{\Q}(a, b) \to (\widehat{\Lib}_{\K}(X), \langle \cdot, \cdot \rangle), a \mapsto \psi, b \mapsto \phi$.
We then define $\exp_{\circledast}^{\K} : \widehat{\Lib}_{\K}(X) \to \G(\KX)$ to be the exponential map; it intertwines $\mathrm{cbh}_{\langle\cdot, \cdot\rangle}$ and $\circledast$. The following proposition recalls from \cite{Rac}, §3.1.8 and \cite{DeGo}, Remark 5.14, the explicit form of $\exp_{\circledast}^{\K}$ as well as gives a proof of this statement.

\begin{proposition}
    For a commutative $\Q$-algebra $\K$ and $\psi \in \widehat{\Lib}_{\K}(X)$, we have
    \begin{enumerate}[label=(\alph*), leftmargin=*]
        \item The exponential map $\exp_{\circledast}^{\K} : \widehat{\Lib}_{\K}(X) \to \G(\KX)$ is a bijection;
        \item $S_{\exp_{\circledast}^{\K}(\psi)} = \exp(s_{\psi})$; where $\psi \mapsto s_{\psi}$ is the map $\widehat{\Lib}_{\K}(X) \to \mathrm{End}_{\K-\Mod}(\KX)$ obtained from the map $\widehat{\Lib}(X) \to \mathrm{End}_{\Q}(\QX)$ in (\ref{eq_s_psi}) by tensoring with $\K$ and $\exp$ is the usual exponential of an endomorphism;    
        \item $\exp_{\circledast}^{\K}(\psi) = \exp(s_{\psi})(1)$. 
    \end{enumerate}
\end{proposition}
\begin{proof}
    \begin{enumerate}[label=(\alph*), leftmargin=*]
        \item See \cite{EF0} §4.1.4 and §4.1.5;
        \item The assignment $\K \mapsto \mathrm{Aut}_{\K-\Mod}(\KX)$ is an affine $\Q$-group scheme and the map $\G(\KX) \to \mathrm{Aut}_{\K-\Mod}(\KX), \Psi \mapsto S_{\Psi}$ defines an affine $\Q$-groupe scheme morphism from $\K \mapsto \G(\KX)$ to $\K \mapsto \mathrm{Aut}_{\K-\Mod}(\KX)$. The associated $\Q$-Lie algebra morphism is $\LibX \to \mathrm{End}_{\Q}(\QX), \psi \mapsto s_{\psi}$. As a consequence, for any $\psi \in \widehat{\Lib}_{\K}(X)$, $S_{\exp_{\circledast}^{\K}(\psi)} = \exp(s_{\psi})$.
        \item Follows by applying the latter equality to $1$, using the identity $S_{\Psi}(1) = \Psi$ for any $\Psi \in \G(\KX)$.
    \end{enumerate}
\end{proof}

To conclude this part, let us notice that the bijection of the map $\exp_{\circledast} : \widehat{\Lib}_{\K}(X) \to \G(\KX)$ implies that we have an identification between the group actions defined in §\ref{double_shuffle_group} with the exponential of the Lie algebra actions of the current section.
	\section{A crossed product formulation of the double shuffle theory} \label{crossed_product}
We construct a crossed product version of the double shuffle formalism.
The relevant algebras and modules are introduced in §\ref{algebra_VG} : (a) an algebra $\hat{\V}_G$ defined by generators and relations, which is then identified with a crossed product algebra involving Racinet's formal series algebra $\KX$; (b) a bialgebra $(\hat{\W}_G, \hat{\Delta}^{\W}_G)$ isomorphic to the bialgebra $(\KY, \hat{\Delta}^{\alg}_{\star})$, where $\hat{\W}_G$ is a subalgebra of $\hat{\V}_G$; (c) a coalgebra $(\hat{\M}_G, \hat{\Delta}^{\M}_G)$ isomorphic to the coalgebra $(\KX / \KX x_0, \hat{\Delta}^{\Mod}_{\star})$, where $\hat{\M}_G$ has a $\hat{\V}_G$-module structure inducing a free rank one $\hat{\W}_G$-module structure on it, compatible with the coproducts $\hat{\Delta}^{\W}_G$ and $\hat{\Delta}^{\M}_G$.   
In §\ref{actions_on_crossed_product} and §\ref{gamma_actions_on_crossed_product}, we construct actions of the group $(\G(\KX), \circledast)$ on these objects by algebra and module automorphisms. This leads us in §\ref{Stabs} to define the stabilizer groups of the coproducts $\hat{\Delta}^{\W}_G$ and $\hat{\Delta}^{\M}_G$ and show in Theorem \ref{Stab_Inclusion} that the stabilizer of the latter is included in the stabilizer of the former.   
\subsection{The algebra \texorpdfstring{$\hat{\V}_G$}{VG}, the bialgebra \texorpdfstring{$(\hat{\W}_G, \hat{\Delta}^{\W}_G)$}{WGDWG} and the coalgebra \texorpdfstring{$(\hat{\M}_G, \hat{\Delta}^{\M}_G)$}{MGDMG}} \label{algebra_VG}
\subsubsection{The algebras $\hat{\V}_G$ and $\hat{\W}_G$ and the module $\hat{\M}_G$}
Let $\hat{\V}_G^{\K}$ (or simply $\hat{\V}_G$ if there is ambiguity) the complete graded topological $\K$-algebra generated by\footnote{The notation $e_0$ and $e_1$ is inspired by \cite{EF1} which in turn is inspired by \cite{DT}.} $\{e_0, e_1\} \sqcup G$ where $e_0$ and $e_1$ are of degree $1$ and elements $g \in G$ are of degree $0$ satisfying the relations:
\begin{multicols}{3}
    \begin{enumerate}[label=(\roman*)]
        \item $g \times h = g h$;
        \item $1 = 1_G$;
        \item $g \times e_0 = e_0 \times g$;
    \end{enumerate}
\end{multicols}
\noindent for any $g, h \in G$; where “$\times$" is the algebra multiplication which we will no longer denote if there is no ambiguity.

Set $\hat{\W}_G^{\K} := \K \oplus \hat{\V}_G e_1$ (or simply $\hat{\W}_G$ if there is ambiguity). It is a graded topological $\K$-subalgebra of $\hat{\V}_G$. Next, the quotient $\hat{\M}_G^{\K} := \hat{\V}_G \Big/ \left(\hat{\V}_G e_0 + \sum_{g \in G \backslash \{1\}} \hat{\V}_G (g-1)\right)$ (or simply $\hat{\M}_G$ if there is ambiguity) is a topological $\K$-module. It is also a topological $\hat{\V}_G$-module and, by restriction, a topological $\hat{\W}_G$-module. Let $1_{\M}$ be the class of $1 \in \hat{\V}_G$ in $\hat{\M}_G$. The map $- \cdot 1_{\M} : \hat{\V}_G \to \hat{\M}_G$ is a surjective topological $\K$-module morphism whose kernel is $\hat{\V}_G e_0 + \sum_{g \in G \backslash \{1\}} \hat{\V}_G (g-1)$.

\subsubsection{The algebra $\hat{\V}_G$ as a crossed product}
First, let us introduce the basic material about the crossed product of an algebra by a group acting by algebra automorphisms.
\begin{definition}
    \label{Def_Crossed_Product}
    Let $A$ be a $\K$-algebra such that the group $G$ acts on $A$ by $\K$-algebra automorphisms. Let us denote $G \times A \ni (g, a) \mapsto a^g \in A$ this action. The \emph{crossed product algebra} of the $\K$-algebra $A$ by the group $G$ denoted $A \rtimes G$ is the $\K$-algebra $(A \otimes \K G, \ast)$ where $\ast$ is the product given by
    \begin{equation}
        \sum_{g \in G} (a_g \otimes g) \ast \sum_{h \in G} (b_h \otimes h) := \sum_{k \in G} \left( \sum_{g,h \in G | gh = k} a_g \, b_h^g \right) \otimes k,
    \end{equation}
    for $a_g, b_g \in A$ with $g \in G$ (\cite{Bou}, Chapter 3, Page 180, Exercise 11).
\end{definition}

\begin{proposition}[Universal property of the crossed product algebra]
    For any $\K$-algebra $B$, there is a natural bijection between the set $\Mor_{\K-\alg}(A \rtimes G, B)$ and the set of pairs $(f, \tau) \in \Mor_{\K-\alg}(A, B) \times \Mor_{\mathrm{grp}}(G, B^{\times})$ such that $f(a^g) = \tau(g) f(a) \tau(g)^{-1}$. 
\end{proposition}
\begin{proof}
    Indeed, given a $\K$-algebra morphism $\beta : A \rtimes G \to B$ we consider:
    \begin{itemize}
        \item The $\K$-algebra morphism $f : A \to B$ given for any $a \in A$ by $f(a)= \beta(a \otimes 1)$;
        \item The group morphism $\tau : G \to B^{\times}$ given for any $g \in G$ by $\tau(g) = \beta(1 \otimes g)$.
    \end{itemize}
    These morphisms verify:
    \begin{align*}
        \tau(g)f(a)\tau(g^{-1}) = & \beta(1\otimes g) \beta(a \otimes 1) \beta(1 \otimes g^{-1})
        = \beta((1\otimes g) \ast (a \otimes 1) \ast (1 \otimes g^{-1})) \\
        = & \beta((a^g \otimes g) \ast (1 \otimes g^{-1}))
        = \beta(a^g \otimes 1)
        = f(a^g).
    \end{align*}
    This shows that the map $\beta \mapsto (f, \tau)$ is well defined. Now let us define a converse map in order to get a bijection. Given any pair $(f, \tau)$ of morphisms satisfying the conditions of the proposition, we set $\beta : a \otimes g \mapsto f(a) \tau(g)$ for any $a \otimes g \in A \rtimes G$. This is a $\K$-algebra morphism. Indeed, for any $a \otimes g$ and $b \otimes h \in A \rtimes G$
    \begin{align*}
        \beta((a \otimes g) \ast (b \otimes h)) =& \beta(a b^g \otimes gh) = f(ab^g) \tau(gh) = f(a) f(b^g) \tau(g) \tau(h) \\
        = & f(a) \tau(g) f(b) \tau(g)^{-1} \tau(g) \tau(h) = f(a) \tau(g) f(b) \tau(h) \\
        = & \beta(a \otimes g) \beta(b \otimes h).
    \end{align*}
    Thus the map $(f, \tau) \to \beta$ is also well defined.
    Finally, one can easily check that the composition of the two maps on both sides gives the identity.
\end{proof}

Now, recall that $g \mapsto t_g$ defines an action of $G$ on $\KX$ by $\K$-algebra automorphisms (\cite{Rac}, §3.1.1). We can then consider the crossed product algebra $\KX \rtimes G$ for this action.

\begin{proposition} \ 
    \begin{enumerate}[label=(\alph*), leftmargin=*]
        \item There is a unique $\K$-algebra morphism $\hat{\V}_G \overset{\alpha}{\to} \KX \rtimes G$ such that $e_0 \mapsto x_0 \otimes 1$, $e_1 \mapsto -x_1 \otimes 1$ and $g \mapsto 1 \otimes g$. \label{VGtoKXG}
        \item There is a unique $\K$-algebra morphism $\KX \rtimes G \overset{\beta}{\to} \hat{\V}_G$ such that $x_0 \otimes 1 \mapsto e_0$ and for $g \in G$, $x_g \otimes 1 \mapsto -g e_1 g^{-1}$ and $1 \otimes g  \mapsto g$. \label{KXGtoVG}
        \item The morphisms $\alpha$ and $\beta$ given respectively in \ref{VGtoKXG} and \ref{KXGtoVG} are isomorphisms which are inverse of one another.
    \end{enumerate}
    \label{isoVGetKXG}
\end{proposition}  
\begin{proof} \ 
    \begin{enumerate}[label=(\alph*), leftmargin=*]
        \item We verify that the images by the morphism $\alpha$ of the generators of $\hat{\V}_G$ satisfy the relations of $\hat{\V}_G$.
        \begin{itemize}[leftmargin=*]
            \item $\alpha(1_G) = 1 \otimes 1_G = \alpha(1)$;
            \item For $g, h \in G, \, \alpha(g) \ast \alpha(h) = (1 \otimes g) \ast (1 \otimes h) = 1 \, t_g(1) \otimes gh = 1 \otimes gh = \alpha(gh)$.
            \item For $g \in G, \alpha(g) \ast \alpha(e_0) = (1 \otimes g) \ast (x_0 \otimes 1) = 1 \, t_g(x_0) \otimes g = x_0 \otimes g$. On the other hand, we have $\alpha(e_0) \ast \alpha(g) = (x_0 \otimes 1) \ast (1 \otimes g) = x_0 \, t_1(1) \otimes g = x_0 \otimes g$. Thus $\alpha(g) \ast \alpha(e_0) = \alpha(e_0) \ast \alpha(g)$.
        \end{itemize}
        \item First, since for any $g \in G$, the element $-g e_1 g^{-1}$ is of degree $1$, there is a unique $\K$-algebra morphism $f : \KX \to \hat{\V}_G$ such that $x_0 \mapsto e_0, \, x_g \mapsto -g e_1 g^{-1}$. Second, there is a unique group morphism $\tau : G \to \hat{\V}_G^{\times}$ given by $g \mapsto g$. Next, for any $g \in G$, the maps $\KX \to \hat{\V}_G$ defined by $a \mapsto f(t_g(a))$ and $a \mapsto \tau(g) f(a) \tau(g)^{-1}$ are $\K$-algebra morphisms that are equal by restriction on generators $x_h$ ($h \in \{0\} \sqcup G$) of $\KX$. Indeed,
        \[
            \tau(g) f(x_0) \tau(g)^{-1} = g e_0 g^{-1} = e_0 g g^{-1} = e_0 = f(x_0) = f(t_g(x_0))
        \]
        and for $h \in G$,
        \[
            \tau(g) f(x_h) \tau(g)^{-1} = g (-h e_1 h^{-1}) g^{-1} = gh e_1 (gh)^{-1} = f(x_{gh}) = f(t_g(x_h)).
        \]
        We then have for any $g \in G$ and any $a \in \KX$, $f(t_g(a)) = \tau(g) f(a) \tau(g)^{-1}$.
        Finally, according to the universal property of crossed products, the pair $(f, \tau)$ gives a unique $\K$-algebra morphism $\beta : \KX \rtimes G \to \hat{\V}_G, a \otimes g \mapsto f(a) \tau(g)$ which verifies $\beta(x_0 \otimes 1) = f(x_0) \tau(1) = e_0$, $\beta(x_g \otimes 1) = f(x_g) \tau(1) = - g e_1 g^{-1}$ and $\beta(1 \otimes g) = f(1) \tau(g) = g$, for $g \in G$.
        \item It is enough to show that the compositions of $\alpha$ and $\beta$ gives the identity. First, since $\beta \circ \alpha : \hat{\V}_G \to \hat{\V}_G$, it is enough to compute it on generators. We have $e_0 \mapsto x_0 \otimes 1 \mapsto e_0$, $e_1 \mapsto -x_1 \otimes 1 \mapsto e_1$ and $g \mapsto 1 \otimes g \mapsto g$. Thus $\beta \circ \alpha = \mathrm{id}_{\hat{\V}_G}$. \newline
        For the converse, we show that $\alpha \circ \beta \in \mathrm{Mor}_{\K-\alg}\left(\KX \rtimes G, \KX \rtimes G\right)$ and the identity of $\KX \rtimes G$ have the same image via the bijection of the universal property of crossed products.
        The image of the identity is the pair
        \[
            f_{\mathrm{id}} : a \mapsto a \otimes 1 \text{ and } \tau_{\mathrm{id}}(g) = 1 \otimes g
        \]
        Next, let us compute the image of $\alpha \circ \beta$. The $\K$-algebra morphism $f$ is given for any $a \in \KX$ by
        \[
            f(a) = \alpha \circ \beta(a \otimes 1)
        \]
        Since it is a $\K$-algebra morphism, it is enough to determine it on $x_g$, $g \in \{0\} \sqcup G$. We have
        \[
            f(x_0) = \alpha \circ \beta(x_0 \otimes 1) = \alpha(e_0) = x_0 \otimes 1     
        \]
        and for $g \in G$,
        \begin{align*}
            \hspace{1cm} f(x_g) = & \alpha \circ \beta(x_g \otimes 1) = \alpha(-ge_1g^{-1}) = - \alpha(g) \ast \alpha(e_1) \ast \alpha(g^{-1}) \\
            \hspace{1cm} = & - (1 \otimes g) \ast (-x_1 \otimes 1) \ast (1 \otimes g^{-1}) = (t_g(x_1) \otimes g) \ast (1 \otimes g^{-1}) = x_g \otimes 1.
        \end{align*}
        We then deduce that for any $a \in \KX$, $f(a) = a \otimes 1$. Next, the group morphism $\tau : G \to (\KX \rtimes G)^{\times}$ is given for any $g \in G$ by
        \[
            \tau(g) = \alpha \circ \beta(1 \otimes g) = \alpha(g) = 1 \otimes g.
        \]
        Finally, by uniqueness of the images we conclude that $\alpha \circ \beta = \mathrm{id}_{\KX \rtimes G}$.
    \end{enumerate}
\end{proof}

\subsubsection{The bialgebra $(\hat{\W}_G, \hat{\Delta}^{\W}_G)$ and the coalgebra $(\hat{\M}_G, \hat{\Delta}^{\M}_G)$}

\begin{proposition}
    The family $\left(e_0^{n_1-1}g_1e_1 \cdots e_0^{n_r-1}g_re_1 e_0^{n_{r+1}-1}g_{r+1}\right)_{\substack{r \in \N, n_1, \dots, n_{r+1}\in \N^{\ast}, \\ g_1, \dots, g_{r+1} \in G}}$ is a basis of the $\K$-module $\hat{\V}_G$.
    \label{basis}
\end{proposition}
\begin{proof}
    Since the family $\left( (-1)^r x_0^{n_1-1} x_{g_1} \cdots x_0^{n_r-1} x_{g_1 \cdots g_r} x_0^{n_{r+1}-1} \right)_{\substack{r \in \N, n_1, \dots, n_{r+1} \in \N^{\ast} \\ g_1, \dots, g_r \in G}}$ is a basis of the $\K$-module $\KX$, it follows that the family
    \[
        \left( (-1)^r x_0^{n_1-1} x_{g_1} \cdots x_0^{n_r-1} x_{g_1 \cdots g_r} x_0^{n_{r+1}-1} \otimes g_1 \cdots g_r g_{r+1} \right)_{\substack{r, n_1, \dots, n_{r+1} \in \N \\ g_1, \dots, g_{r+1} \in G}}
    \]
    is a basis of the $\K$-module $\KX \otimes \K G$.
    Thus, its image by the bijection $\beta$ (given in Proposition \ref{isoVGetKXG} \ref{KXGtoVG}) is a basis of $\hat{\V}_G$. Moreover, for $r \in \N, n_1, \dots, n_{r+1} \in \N^{\ast}$ and $g_1, \dots, g_{r+1} \in G$, we have
    \begin{align*}
        &x_0^{n_1-1} x_{g_1} \cdots x_0^{n_r-1} x_{g_1 \cdots g_r} x_0^{n_{r+1}-1} \otimes g_1 \cdots g_r g_{r+1} = \\
        &(x_0^{n_1-1} \otimes 1) \ast (x_{g_1} \otimes 1) \ast \cdots \ast (x_0^{n_r-1} \otimes 1) \ast (x_{g_1 \cdots g_r} \otimes 1) \ast \\ & (x_0^{n_{r+1}-1} \otimes 1) \ast (1 \otimes g_1) \ast \cdots \ast (1 \otimes g_r) \ast (1 \otimes g_{r+1})
    \end{align*}
    Then
    \begin{align}
        \label{beta_of_basis}
        &\beta((-1)^rx_0^{n_1-1} x_{g_1} \cdots x_0^{n_r-1} x_{g_1 \cdots g_r} x_0^{n_{r+1}-1} \otimes g_1 \cdots g_r g_{r+1}) \\ \notag
        & = (-1)^r \beta(x_0^{n_1-1} \otimes 1) \beta(x_{g_1} \otimes 1) \cdots \beta(x_0^{n_r-1} \otimes 1) \beta(x_{g_1 \cdots g_r} \otimes 1) \\ \notag & \quad \, \, \beta(x_0^{n_{r+1}-1} \otimes 1) \beta(1 \otimes g_1) \cdots \beta(1 \otimes g_r) \beta(1 \otimes g_{r+1}) \\ \notag
        & = e_0^{n_1-1} g_1 e_1 g_1^{-1} \cdots e_0^{n_r-1} g_1 \cdots g_r e_1 g_1^{-1} \cdots g_r^{-1} e_0^{n_{r+1}-1} g_1 \cdots g_r g_{r+1} \\ \notag
        & = e_0^{n_1-1} g_1 e_1 \cdots e_0^{n_r-1} g_1^{-1} \cdots g_{r-1}^{-1}g_1 \cdots g_r e_1 e_0^{n_{r+1}-1} g_1^{-1} \cdots g_r^{-1} g_1 \cdots g_r g_{r+1} \\ \notag
        & = e_0^{n_1-1} g_1 e_1 \cdots e_0^{n_r-1} g_r e_1 e_0^{n_{r+1}-1} g_{r+1}
    \end{align}
    This gives us the wanted result.
\end{proof}

\begin{proposition} \ 
    \label{WG_free_algebra}
    \begin{enumerate}[label=(\alph*), leftmargin=*]
        \item The family $\{1\} \cup \left(e_0^{n_1-1}g_1e_1 \cdots e_0^{n_r-1}g_re_1 e_0^{n_{r+1}-1}g_{r+1}e_1\right)_{\substack{r \in \N, n_1, \dots, n_r, n_{r+1} \in \N^{\ast}, \\ g_1, \dots, g_r, g_{r+1} \in G}}$ is a basis of the $\K$-module $\hat{\W}_G$. \label{basisW}
        \item \label{iso_WG_KZ} The $\K$-subalgebra $\hat{\W}_G$ is freely generated by the family $$Z=\{z_{n,g}:= -e_0^{n-1}ge_1 \, | \, (n,g) \in \N^{\ast} \times G\}.$$
    \end{enumerate}
\end{proposition}
\begin{proof} \ 
    \begin{enumerate}[label=(\alph*), leftmargin=*]
        \item First, $\hat{\W}_G$ is the image of the $\K$-module morphism $\K \oplus \hat{\V}_G \to \hat{\V}_G, (\lambda,v) \mapsto \lambda + v e_1$. Second, according to Proposition \ref{basis}, the family
        \[
            (1, 0), \left(0, e_0^{n_1-1}g_1e_1 \cdots e_0^{n_r-1}g_re_1 e_0^{n_{r+1}-1}g_{r+1} \right)_{\substack{r \in \N, n_1, \dots, n_r, n_{r+1} \in \N^{\ast}, \\ g_1, \dots, g_r, g_{r+1} \in G}}
        \]
        is a basis of the $\K$-module $\K \oplus \hat{\V}_G$. Moreover, the image of this basis by this $\K$-module morphism is the family
        \[
            \{1\} \cup \left(e_0^{n_1-1}g_1e_1 \cdots e_0^{n_r-1}g_re_1 e_0^{n_{r+1}-1}g_{r+1}e_1\right)_{\substack{r \in \N, n_1, \dots, n_r, n_{r+1} \in \N^{\ast}, \\ g_1, \dots, g_r, g_{r+1} \in G}}
        \]
        which is free since it is contained in a basis of the target. This implies that this family is a basis of the image of the previous morphism which is $\hat{\W}_G$.
        \item Let $\K\langle\langle Z \rangle\rangle$ be the free algebra over the letters $z_{n,g}$, $(n \in \N^{\ast}, g \in G)$, which we view as free variables. Then there is a unique $\K$-algebra morphism $\K \langle\langle Z \rangle \rangle \to \hat{\W}_G$
        given by $z_{n,g} \mapsto -e_0^{n-1}ge_1$. Let us show that it is an isomorphism: \newline
        The free $\K$-module $\K \langle \langle Z \rangle \rangle$ has basis $\{1\} \cup (z_{n_1, g_1} \cdots z_{n_{r+1},g_{r+1}})_{\substack{r \in \N, n_1, \dots, n_{r+1} \in \N^{\ast} \\ g_1, \dots, g_{r+1} \in G}}$ and, as a $\K$-module, $\hat{\W}_G$ has basis $\{1\} \cup \left(e_0^{n_1-1}g_1e_1 \cdots e_0^{n_{r+1}-1}g_{r+1}e_1 \right)_{\substack{r \in \N, n_1, \dots, n_{r+1} \in \N^{\ast}, \\ g_1, \dots, g_{r+1} \in G}}$ according to \ref{basisW}.
        One computes the image by $z_{n,g} \mapsto -e_0^{n-1}ge_1$ of the latter basis and finds it to be equal to the former basis. Therefore, $z_{n,g} \mapsto -e_0^{n-1}ge_1$ induces a bijection between the two basis and then a bijection between $\K \langle \langle Z \rangle \rangle$ and $\hat{\W}_G$. Hence, $z_{n,g} \mapsto -e_0^{n-1}ge_1$ is a $\K$-algebra isomorphism between $\K \langle \langle Z \rangle \rangle$ and $\hat{\W}_G$.  
    \end{enumerate}
\end{proof}
\noindent So, from now on, by abuse of notation, we will identify elements of $\hat{\W}_G$ with elements of $\KZ$ by the $\K$-algebra isomorphism $z_{n,g} \mapsto -e_0^{n-1}ge_1$.

\begin{proposition}
    There exists a $\K$-module isomorphism $\KX/\KX x_0 \overset{\kappa}{\to} \hat{\M}_G$ uniquely determined by the condition that the diagram
    \begin{equation}
        \label{diag_iso_MG}
        \begin{tikzcd}
            \KX \ar["\beta \circ (- \otimes 1)"]{rr} \ar["\pi_Y"']{d} & & \hat{\V}_G \ar["- \cdot 1_{\M}"]{d} \\
            \KX/\KX x_0 \ar["\kappa"']{rr} & & \hat{\M}_G
        \end{tikzcd}
    \end{equation}
    commutes.
    \label{isoMG}
\end{proposition}
\noindent We will prove this proposition by using the following general lemma. In this lemma, for any $\K$-module $M$ and any $\K$-submodule $M'$, let us denote $\mathrm{can}_{M, M'} : M \to M/M'$ the canonical projection.  
\begin{lemma}
    Let $f : M \to N$ a $\K$-module morphism. Let $M'$ a submodule of $M$ and $N', N''$ two submodules of $N$ such that
    \begin{enumerate}[label=(\alph*), leftmargin=*]
        \item \label{double_inclusion} $f(M') \subset N' \subset f(M') + N''$ and,
        \item \label{iso_compo} $\mathrm{can}_{N, N''} \circ f$ is an isomorphism,
    \end{enumerate}
    Then, there is a unique $\K$-module morphism $\bar{f} : M/M' \to N/(N'+N'')$ such that the diagram
    \begin{equation}
        \begin{tikzcd}
            M \ar["f"]{r} \ar["\mathrm{can}_{M,M'}"']{d} & N \ar["\mathrm{can}_{N,N'+N''}"]{d} \\
            M/M' \ar["\bar{f}"']{r} & N/(N'+N'')
        \end{tikzcd}
        \label{diag_iso}
    \end{equation}
    commutes. Moreover, $\bar{f}$ is a $\K$-module isomorphism.
    \label{abstract_lemma}
\end{lemma}
\begin{proof}
    First, $\bar{f} : M/M' \to N/(N'+N'')$ is well defined since $f(M') \subset N' \subset N'+N''$. Next, let us show that this $\K$-module morphism is an isomorphism.
    \begin{description}[leftmargin=0pt]
        \item[Injectivity]
        Let $\mu \in M/M'$ be such that $\bar{f}(\mu)=0$. Let $m \in M$ be such that $\mu = \mathrm{can}_{M, M'}(m)$. By the commutativity of the diagram (\ref{diag_iso}), the assumption on $\mu$ implies that $f(m) \in N' + N''$. But, since $N' \subset f(M') + N''$ we get $f(m) \in f(M') + N''$. Therefore there exists $m' \in M'$ such that $f(m) \in f(m')+N''$ then $f(m-m') \in N''$. This means that $\mathrm{can}_{N, N''} \circ f (m-m') = 0$. Finally, since $\mathrm{can}_{N, N''} \circ f$ is an isomorphism, this implies that $m=m'$ (elements of $M$). Since $m' \in M'$, this implies that $m \in M'$. Therefore, $\mu = \mathrm{can}_{M,M'}(m) = 0 \in M/M'$.
        \item[Surjectivity]
        Diagram (\ref{diag_iso}) can be extended to the commutative diagram
        \[\begin{tikzcd}
            M \ar["f"]{r} \ar["\mathrm{can}_{M,M'}"']{d} & N \ar["\mathrm{can}_{N,N'+N''}"]{d} \ar["\mathrm{can}_{N, N''}"]{rr} & & N/N'' \ar["\mathrm{can}_{N/N'', N'/(N'\cap N'')}"]{d}\\
            M/M' \ar["\bar{f}"']{r} & N/(N'+N'') \ar["\mathrm{iso}_{N, N', N''}"']{rr} & & (N/N'')/(N'/(N'\cap N''))
        \end{tikzcd}\]
        where $\mathrm{iso}_{N, N', N''} : N/(N'+N'') \to (N/N'')/(N'/(N'\cap N''))$ is the canonical isomorphism. We then obtain
        \begin{align*}
            N/(N'+N'') = & \mathrm{iso}_{N, N', N''}^{-1}((N/N'')/(N'/(N'\cap N''))) \\
            = & \mathrm{iso}_{N, N', N''}^{-1} \circ \mathrm{can}_{N/N'', N'/(N' \cap N'')}(N/N'') \\   
            = & \mathrm{iso}_{N, N', N''}^{-1} \circ \mathrm{can}_{N/N'', N'/(N' \cap N'')} \circ \mathrm{can}_{N, N''} \circ f (M) \\
            = & \mathrm{iso}_{N, N', N''}^{-1} \circ \mathrm{iso}_{N, N', N''} \circ \bar{f} \circ \mathrm{can}_{M, M'}(M) = \bar{f}(M/M')
        \end{align*}
        where the first equality comes from the fact that $\mathrm{iso}_{N, N', N''}$ is $\K$-module isomorphism; the second one from the fact that $\mathrm{can}_{N/N'', N'/(N' \cap N'')}$ is a surjective $\K$-module morphism; the third one from the fact that $\mathrm{can}_{N, N''} \circ f$ is a $\K$-module isomorphism; the fourth one from the commutativity of the external square; and the fifth one from the fact that $\mathrm{can}_{M, M''}$ is a surjective $\K$-module morphism.  
    \end{description}
\end{proof}

\begin{proof}[Proof of Proposition \ref{isoMG}]
    This is done by application of Lemma \ref{abstract_lemma} for $M=\KX$, $N=\hat{\V}_G$, $M'=\KX x_0$, $N'=\hat{\V}_G e_0$, $N''=\sum_{g \in G \backslash \{1\}} \hat{\V}_G (g-1)$ and $f=\beta \circ (- \otimes 1)$. It, therefore, suffices to prove that criteria \ref{double_inclusion} and \ref{iso_compo} of Lemma \ref{abstract_lemma} are satisfied.  
    \begin{description}[leftmargin=0pt]
        \item[Criterion \ref{double_inclusion}] $\beta(\KX x_0 \otimes 1) \subset \hat{\V}_G e_0 \subset \beta(\KX x_0 \otimes 1) + \sum_{g \in G \backslash \{1\}} \hat{\V}_G (g-1)$. \newline
        For the first inclusion, we have for any $a \in \KX$
        \[
            \beta(a x_0 \otimes 1) = \beta(a \otimes 1) \beta(x_0 \otimes 1) = \beta(a \otimes 1) e_0 \in \hat{\V}_G e_0.  
        \] 
        Therefore, $\beta(\KX x_0 \otimes 1) \subset \hat{\V}_G e_0$. \newline
        For the second inclusion, by using the basis of $\hat{\V}_G$ described in Proposition \ref{basis}, we have for $r \in \N$, $n_1, \dots, n_{r+1} \in \N^{\ast}$ and $g_1, \dots, g_{r+1} \in G$,
        \begin{align*} 
            &\left(e_0^{n_1-1} g_1 e_1 \cdots e_0^{n_r-1} g_r e_1 e_0^{n_{r+1}-1} g_{r+1}\right) e_0 = e_0^{n_1-1} g_1 e_1 \cdots e_0^{n_r-1} g_r e_1 e_0^{n_{r+1}} g_{r+1} \\
            &= (-1)^r \beta(x_0^{n_1-1} x_{g_1} \cdots x_0^{n_r-1} x_{g_1 \cdots g_r} x_0^{n_{r+1}} \otimes g_1 \cdots g_{r+1}) \\
            &= (-1)^r \beta\left(\left(x_0^{n_1-1} x_{g_1} \cdots x_0^{n_r-1} x_{g_1 \cdots g_r} x_0^{n_{r+1}-1}\right) x_0 \otimes 1\right) g_1 \cdots g_{r+1} \\
            &= (-1)^r \beta\left(\left(x_0^{n_1-1} x_{g_1} \cdots x_0^{n_r-1} x_{g_1 \cdots g_r} x_0^{n_{r+1}-1}\right) x_0 \otimes 1\right) \\
            &+ (-1)^r \beta\left(\left(x_0^{n_1-1} x_{g_1} \cdots x_0^{n_r-1} x_{g_1 \cdots g_r} x_0^{n_{r+1}-1}\right) x_0 \otimes 1\right) (g_1 \cdots g_{r+1} - 1) 
        \end{align*}
        where the first equality comes from the relation $g e_0 = e_0 g$ for any $g \in G$; the second one from computation (\ref{beta_of_basis}) and the third one from the fact that $a x_0 \otimes g = (a x_0 \otimes 1) \ast (1 \otimes g)$ for any $a \in \KX$ and any $g \in G$. Finally, the last equality shows that we obtain an element of $\beta(\KX x_0 \otimes 1) + \sum_{g \in G \backslash \{1\}} \hat{\V}_G (g-1)$, thus proving the claimed inclusion.
        \item[Criterion \ref{iso_compo}] The map $\mathrm{can}_{\hat{\V}_G, \hspace{-0.2cm} \underset{g \in G \backslash \{1\}}{\sum} \hspace{-0.2cm} \hat{\V}_G (g-1)} \circ \beta \circ ( - \otimes 1) : \KX \to \hat{\V}_G \Big/ \left(\underset{g \in G \backslash \{1\}}{\sum} \hspace{-0.2cm} \hat{\V}_G (g-1) \right)$ is an isomorphism. \newline
        Let us consider the commutative diagram
        \begin{equation}
            \begin{tikzcd}
                \KX \otimes \bigoplus_{g \in G \backslash \{1\}} \K G \ar["\mathrm{id} \otimes \sigma"]{r} \ar[]{d} & \KX \otimes \K G \ar[equal]{d}\\
                \bigoplus_{g \in G \backslash \{1\}} \left(\KX \otimes \K G\right) \ar[]{r} \ar["\bigoplus_{g \in G \backslash \{1\}} \beta"']{d} & \KX \otimes \K G \ar["\beta"]{d} \\
                \bigoplus_{g \in G \backslash \{1\}} \hat{\V}_G \ar["\Sigma"']{r} & \hat{\V}_G
            \end{tikzcd}
        \end{equation}
        where in the horizontal arrows we have the $\K$-module morphisms
        \[
            \Sigma : \bigoplus_{g \in G \backslash \{1\}} \hat{\V}_G \longrightarrow \hat{\V}_G, \quad (v_g)_{g \in G \backslash \{1\}} \longmapsto \underset{g \in G \backslash \{1\}}{\sum} v_g (g-1)
        \]
        and
        \[
            \sigma : \bigoplus_{g \in G \backslash \{1\}} \K G \longrightarrow \K G, \quad (h_g)_{g \in G \backslash \{1\}} \longmapsto \underset{g \in G \backslash \{1\}}{\sum} h_g (g-1).
        \]
        Since the vertical arrows are isomorphisms, they induce an isomorphism beteween the cokernels of the top and bottom morphisms. We can then extend the above diagram in the following way
        \begin{equation}
            \begin{tikzcd}
                \KX \otimes \bigoplus_{g \in G \backslash \{1\}} \K G \ar["\mathrm{id} \otimes \sigma"]{r} \ar[]{d} & \KX \otimes \K G \ar[]{r} \ar[equal]{d} & \mathrm{coker}(\mathrm{id} \otimes \sigma) \ar[]{dd} \\
                \bigoplus_{g \in G \backslash \{1\}} \left(\KX \otimes \K G\right) \ar[]{r} \ar["\bigoplus_{g \in G \backslash \{1\}} \beta"']{d} & \KX \otimes \K G \ar["\beta"]{d} & \\
                \bigoplus_{g \in G \backslash \{1\}} \hat{\V}_G \ar["\Sigma"']{r} & \hat{\V}_G \ar[]{r} & \mathrm{coker}(\Sigma)
            \end{tikzcd}
            \label{diag_with_coker}
        \end{equation}
        On the other hand, we have
        \[
            \mathrm{coker}(\Sigma) = \hat{\V}_G \Bigg/ \left(\sum_{g \in G \backslash \{1\}} \hat{\V}_G (g-1) \right).
        \]
        and
        \[
            \mathrm{coker}(\sigma) =  \K G \Bigg/ \left(\sum_{g \in G \backslash \{1\}} \K G (g-1) \right) \simeq \K.
        \]
        Then
        \[
            \mathrm{coker}(\mathrm{id} \otimes \sigma) \simeq \KX \otimes \mathrm{coker}(\sigma) \simeq \KX \otimes \K \simeq \KX.
        \]
        Thus, the isomorphism between cokernels establishes that $\KX$ is isomorphic to $\hat{\V}_G \Big/ \left(\sum_{g \in G \backslash \{1\}} \hat{\V}_G (g-1) \right)$. Moreover, thanks to the commutativity of diagram (\ref{diag_with_coker}), this isomorphism is exactly $\mathrm{can}_{\hat{\V}_G, \hspace{-0.2cm} \underset{g \in G \backslash \{1\}}{\sum} \hspace{-0.2cm} \hat{\V}_G (g-1)} \circ \beta \circ ( - \otimes 1)$.
    \end{description}
\end{proof}

\begin{corollary} \ 
    \label{iso_WG_MG}
    \begin{enumerate}[label=(\alph*), leftmargin=*]
        \item \label{item_diag_projections} The following diagram
        \begin{equation}
            \label{diag_projections}
            \begin{tikzcd}
                \K\langle\langle Y \rangle\rangle \ar["\varpi"]{rr} \ar["\pi_Y"']{d} && \hat{\W}_G \ar["- \cdot 1_{\M}"]{d} \\
                \KX/\KX \, x_0 \ar["\kappa \circ \overline{\q}^{-1}"']{rr} & & \hat{\M}_G
            \end{tikzcd}
        \end{equation}
        commutes, where $\varpi : \KY \to \hat{\W}_G$ is the $\K$-algebra isomorphism uniquely defined by $y_{n, g} \mapsto z_{n, g}$ and $\overline{\q}$ is the $\K$-module automorphism of $\KX/\KX x_0$ given in §\ref{basic_objects}.
        \item \label{item_iso_WG_MG} The map $- \cdot 1_{\M} : \hat{\W}_G \to \hat{\M}_G$ is a $\K$-module isomorphism and $\hat{\M}_G$ is free of rank $1$ as a $\hat{\W}_G$-module.
    \end{enumerate}
\end{corollary}

\begin{proof} \ 
    \begin{enumerate}[label=(\alph*), leftmargin=*]
        \item One needs to show the equality of two maps from $\KY$ to $\hat{\M}_G$. Since these maps are both $\K$-module morphisms, it is enough to show the equality of the images of the elements of a basis of the source module. Such a basis is $\left(y_{n_1, g_1} \cdots y_{n_r g_r}\right)_{\substack{n_1, \dots, n_r \in \N^{\ast} \\ g_1, \dots, g_r \in G}}$ (\cite{Rac}, §2.2.7.) \newline
        For $r \in \N$, $n_1, \dots, n_r \in \N^{\ast}$ and $g_1, \dots, g_r \in G$ we have
        \[
            (- \cdot 1_{\M}) \circ \varpi (y_{n_1, g_1} \cdots y_{n_r, g_r}) = z_{n_1, g_1} \cdots z_{n_r, g_r} \cdot 1_{\M} 
        \]
        On the other hand,
        \begin{align*}
            &\kappa \circ \overline{\q}^{-1} \circ \pi_Y(y_{n_1, g_1} \cdots y_{n_r, g_r}) = \kappa(x_0^{n_1-1} x_{g_1} \cdots x_0^{n_1-1} x_{g_1 \cdots g_r}) \\
            &= \beta(x_0^{n_1-1} x_{g_1} \cdots x_0^{n_r-1} x_{g_1 \cdots g_r} \otimes 1) \cdot 1_{\M} \\
            &= (-1)^r e_0^{n_1-1} g_1 e_1 \cdots e_0^{n_r-1} g_r e_1 g_1^{-1} \cdots g_r^{-1} \cdot 1_{\M} \\
            &= (-e_0^{n_1-1} g_1 e_1) \cdots (-e_0^{n_r-1} g_r e_1) \cdot 1_{\M} \\
            &= z_{n_1, g_1} \cdots z_{n_r, g_r} \cdot 1_{\M}
        \end{align*}
        where the first equality comes from \cite{Rac}, §2.2.7; the second one from the commutative diagram (\ref{diag_iso_MG}); the third one from computation (\ref{beta_of_basis}) with $n_{r+1}=1$ and $g_{r+1}=(g_1 \cdots g_r)^{-1}$; and the fourth one from the fact that for any $v \in \hat{\V}_G$ and any $g \in G$, we have $v g \cdot 1_{\M} = v \cdot 1_{\M}$.
        \item First, the following maps are $\K$-module isomorphisms:
        \begin{itemize}
            \item $\varpi : \KY \to \hat{\W}_G$ : it sends the basis $(y_{n_1, g_1} \cdots y_{n_r, g_r})_{\substack{r \in \N, n_1, \dots, n_r \in \N^{\ast} \\ g_1, \dots, g_r \in G}}$ of the $\K$-module $\KY$ to the basis\footnote{such a family is a basis of $\hat{\W}_G$ thanks to Proposition \ref{WG_free_algebra}.} $(z_{n_1, g_1} \cdots z_{n_r, g_r})_{\substack{r \in \N, n_1, \dots, n_r \in \N^{\ast} \\ g_1, \dots, g_r \in G}}$ of $\hat{\W}_G$.
            \item $\pi_Y : \KY \to \KX/\KX x_0$ : see \cite{Rac}, §2.2.5.
            \item $\kappa \circ \overline{\q}^{-1} : \KX / \KX x_0 \to \hat{\M}_G$ : see Proposition \ref{isoMG} and \cite{Rac}, §2.2.7. 
        \end{itemize}
        Next, the diagram (\ref{diag_projections}) commutes, thanks to \ref{item_diag_projections}. This allows us to conclude that the map $- \cdot 1_{\M} : \hat{\W}_G \to \hat{\M}_G$ is a $\K$-module isomorphism and that $\hat{\M}_G$ is a free $\hat{\W}_G$-module of rank $1$.
    \end{enumerate}
\end{proof}

\noindent \textbf{Remark.} The composed algebra morphisms $\KY \overset{\varpi}{\longrightarrow} \hat{\W}_G \hookrightarrow \hat{\V}_G$ and $\KY \hookrightarrow \KX \overset{\beta \circ ( - \otimes 1)}{\longrightarrow} \hat{\V}_G$ do not coincide when $G \neq \{1\}$.

Now, we are able to put more structure on $\hat{\W}_G$ and $\hat{\M}_G$. More precisely, we are going to define a coproduct $\hat{\W}_G$ and a coproduct on $\hat{\M}_G$.

\begin{propdef} \ 
    \label{DeltaW_DeltaM}
    \begin{enumerate}[label=(\alph*), leftmargin=*]
        \item \label{item_DeltaW} There exists a unique topological $\K$-algebra morphism $\hat{\Delta}^{\W}_{G} : \hat{\W}_G \to (\hat{\W}_G)^{\otimes 2}$ such that for any $(n, g) \in \N^{\ast} \times G$
        \begin{equation}
            \label{DeltaW}
            \hat{\Delta}^{\W}_{G}(z_{n, g}) = z_{n, g} \otimes 1 + 1 \otimes z_{n, g} + \sum_{\substack{k = 1 \\ h \in G}}^{n-1} z_{k, h} \otimes z_{n-k, hg^{-1}}. 
        \end{equation}
        The pair $(\hat{\W}_G, \hat{\Delta}^{\W}_G)$ is then a topological bialgebra.
        \item \label{item_DeltaM} There exists a unique topological $\K$-module morphism $\hat{\Delta}^{\M}_G : \hat{\M}_G \to (\hat{\M}_G)^{\otimes 2}$ such that the following diagram
        \begin{equation}
            \label{diag_DeltaW_DeltaM}
            \begin{tikzcd}
                \hat{\W}_G \ar["\hat{\Delta}^{\W}_G"]{rrr} \ar["- \cdot 1_{\M}"']{d} & & & (\hat{\W}_G)^{\hat{\otimes} 2} \ar["- \cdot 1_{\M}^{\otimes 2}"]{d} \\
                \hat{\M}_G \ar["\hat{\Delta}^{\M}_G"']{rrr} & & & (\hat{\M}_G)^{\hat{\otimes} 2}
            \end{tikzcd}
        \end{equation}
        commutes. The pair $(\hat{\M}_G, \hat{\Delta}^{\M}_G)$ is then a cocommutative coassociative coalgebra.
        \item \label{compat_DeltaW_DeltaM} For any $w \in \hat{\W}_G$ and any $m \in \hat{\M}_G$ we have
        \begin{equation}
            \hat{\Delta}^{\M}_G(w \cdot m) = \hat{\Delta}^{\W}_G(w) \cdot \hat{\Delta}^{\M}_G(m). 
        \end{equation}
    \end{enumerate}
\end{propdef}
\begin{proof} \ 
    \begin{enumerate}[label=(\alph*), leftmargin=*]
        \item This is a consequence of Proposition \ref{WG_free_algebra}\ref{iso_WG_KZ}. 
        \item This is a consequence of \ref{item_DeltaW} and Corollary \ref{iso_WG_MG}\ref{item_iso_WG_MG}.
        \item Since $- \cdot 1_{\M} : \hat{\W}_G \to \hat{\M}_G$ is a $\K$-module ismorphism, for $m \in \hat{\M}_G$ there exists a unique $w' \in \hat{\W}_G$ such that $m = w' \cdot 1_{\M}$. We then have
        \begin{align*}
            & \hat{\Delta}^{\M}_G(w \cdot m) = \hat{\Delta}^{\M}_G(ww' \cdot 1_{\M}) = \hat{\Delta}^{\W}_G(ww') \cdot 1_{\M}^{\otimes 2} = \hat{\Delta}^{\W}_G(w) \hat{\Delta}^{\W}_G(w') \cdot 1_{\M}^{\otimes 2} \\
            & = \hat{\Delta}^{\W}_G(w) \cdot \left(\hat{\Delta}^{\W}_G(w') \cdot 1_{\M}^{\otimes 2}\right) = \hat{\Delta}^{\W}_G(w) \cdot \hat{\Delta}^{\M}_G(w' \cdot 1_{\M}) = \hat{\Delta}^{\W}_G(w) \cdot \hat{\Delta}^{\M}_G(m)
        \end{align*}
        where the first and the fourth equalities come from the fact that $- \cdot 1_{\M} : \hat{\W}_G \to \hat{\M}_G$ is $\hat{\W}_G$-module morphism; the second and the fifth one from the commutative diagram (\ref{diag_DeltaW_DeltaM}) and the third one from the fact that $\hat{\Delta}^{\W}_G$ is a $\K$-algebra morphism. 
    \end{enumerate}
\end{proof}

\subsection{Actions of the group \texorpdfstring{$(\G(\KX), \circledast)$}{GKX*} by automorphisms} \label{actions_on_crossed_product}
We recall that the map $\beta : \KX \rtimes G \to \hat{\V}_G$ is the $\K$-algebra isomorphism given in Proposition \ref{isoVGetKXG} \ref{KXGtoVG}.
\subsubsection{Actions of \texorpdfstring{$(\G(\KX), \circledast)$}{GKX*} by algebra automorphisms}
\begin{propdef}
    \label{aut0}
    Let $\Psi \in \G(\KX)$. There exists a unique topological $\K$-algebra automorphism $\aut^{\V, (0)}_{\Psi}$ of $\hat{\V}_G$ such that
    \begin{equation}
        e_0 \mapsto e_0; \quad e_1 \mapsto \beta(\Psi^{-1} \otimes 1) e_1 \beta(\Psi \otimes 1); \quad g \mapsto g, \text{ for } g \in G, 
    \end{equation}
\end{propdef}
\begin{proof}
    First, let us verify that the images by the morphism $\aut^{\V, (0)}_{\Psi}$ of the generators of $\hat{\V}_G$ satisfy the relations of $\hat{\V}_G$. Indeed, for $g, h \in G$ we have:
    \begin{itemize}[leftmargin=*]
        \item $\aut^{\V, (0)}_{\Psi}(1_G) = 1_G = 1 = \aut^{\V, (0)}_{\Psi}(1)$;
        \item $\aut^{\V, (0)}_{\Psi}(g) \times \aut^{\V, (0)}_{\Psi}(h) = g \times h = gh = \aut^{\V, (0)}_{\Psi}(gh)$;
        \item $\aut^{\V, (0)}_{\Psi}(g) \times \aut^{\V, (0)}_{\Psi}(e_0) = g \times e_0 = e_0 \times g = \aut^{\V, (0)}_{\Psi}(e_0) \times \aut^{\V, (0)}_{\Psi}(g)$. 
    \end{itemize}
    This proves the existence and uniqueness of the algebra endomorphism $\aut^{\V, (0)}_{\Psi}$.
    Next, in order to prove that $\aut^{\V, (0)}_{\Psi}$ is an automorphism, we show that the diagram
    \begin{equation}
        \label{beta_auts}
        \begin{tikzcd}
            \KX \rtimes G \ar["\beta"]{rrr} \ar["\aut_{\Psi} \otimes \mathrm{id}_{\K G}"']{d} & & & \hat{\V}_G \ar["\aut^{\V, (0)}_{\Psi}"]{d} \\
            \KX \rtimes G \ar["\beta"']{rrr} & & & \hat{\V}_G
        \end{tikzcd}
    \end{equation}
    commutes, where $\aut_{\Psi}$ is the $\K$-algebra automorphism in (\ref{aut_Psi}). Since all arrows of Diagram (\ref{beta_auts}) are $\K$-algebra morphisms, it is enough to check the commutativity on generators:
    \begin{itemize}[leftmargin=*]
        \item $\aut^{\V, (0)}_{\Psi} \circ \beta(x_0 \otimes 1) = \aut^{\V, (0)}_{\Psi}(e_0) = e_0$ and \newline $\beta \circ (\aut_{\Psi} \otimes \mathrm{id}_{\K G})(x_0 \otimes 1) = \beta(\aut_{\Psi}(x_0) \otimes 1) = \beta(x_0 \otimes 1) = e_0$.
        \item For $g \in G$, $\aut^{\V, (0)}_{\Psi} \circ \beta(1 \otimes g) = \aut^{\V, (0)}_{\Psi}(g) = g$ and \newline $\beta \circ (\aut_{\Psi} \otimes \mathrm{id}_{\K G})(1 \otimes g) = \beta(\aut_{\Psi}(1) \otimes g) = \beta(1 \otimes g) = g$.
        \item For $g \in G$, $\aut^{\V, (0)}_{\Psi} \circ \beta(x_g \otimes 1) = \aut^{\V, (0)}_{\Psi}(-ge_1g^{-1}) = -g \beta(\Psi^{-1}\otimes 1) e_1 \beta(\Psi \otimes 1) g^{-1}$ and
        \[\begin{aligned}
            & \beta \circ (\aut_{\Psi} \otimes \mathrm{id}_{\K G})(x_g \otimes 1) = \beta(\aut_{\Psi}(x_g) \otimes 1) = \beta(t_g(\Psi^{-1}) x_g t_g(\Psi) \otimes 1) \\
            = & \beta((1\otimes g) \ast (\Psi^{-1} \otimes 1) \ast (1 \otimes g^{-1}) \ast (x_g \otimes 1) \ast (1 \otimes g) \ast (\Psi \otimes 1) \ast (1 \otimes g^{-1})) \\
            = & \beta(1\otimes g) \beta(\Psi^{-1} \otimes 1) \beta(1 \otimes g^{-1}) \beta(x_g \otimes 1) \beta(1 \otimes g) \beta(\Psi \otimes 1) \beta(1 \otimes g^{-1})) \\
            =& g \beta(\Psi^{-1} \otimes 1) g^{-1} (-ge_1 g^{-1}) g \beta(\Psi \otimes 1) g^{-1} \\
            = & - g \beta(\Psi^{-1} \otimes 1) e_1 \beta(\Psi \otimes 1) g^{-1}.
        \end{aligned}\]
    \end{itemize}
    Finally, $\aut^{\V, (0)}_{\Psi}$ is an automorphism thanks to the commutativity of diagram (\ref{beta_auts}) and the fact that $\beta : \KX \rtimes G \to \hat{\V}_G$ and $\aut_{\Psi} \otimes \mathrm{id}_{\K G} : \KX \rtimes G \to \KX \rtimes G$ are $\K$-algebra isomorphisms.
\end{proof}

\begin{definition}
    For $\Psi \in \G(\KX)$, we define $\aut^{\V, (1)}_{\Psi}$ to be the topological $\K$-algebra automorphism of $\hat{\V}_G$ given by
    \begin{equation}
        \aut^{\V, (1)}_{\Psi} := \Ad_{\beta(\Psi \otimes 1)} \circ \aut^{\V, (0)}_{\Psi},
    \end{equation}
\end{definition}

\begin{proposition}
    \label{group_morph_V} \ 
    \begin{enumerate}[label=(\alph*), leftmargin=*]
        \item \label{act_aut0} There is a group action of $(\G(\KX), \circledast)$ on $\hat{\V}_G$ by $\K$-algebra automorphisms
        \[
            (\G(\KX), \circledast) \longrightarrow \Aut_{\K-\alg}(\hat{\V}_G), \, \Psi \longmapsto \aut^{\V, (0)}_{\Psi}
        \]
        \item \label{act_aut1} There is a group action of $(\G(\KX), \circledast)$ on $\hat{\V}_G$ by $\K$-algebra automorphisms
        \[
            (\G(\KX), \circledast) \longrightarrow \Aut_{\K-\alg}(\hat{\V}_G), \, \Psi \longmapsto \aut^{\V, (1)}_{\Psi}
        \]
    \end{enumerate}
\end{proposition}

\begin{proof} \, 
    \begin{enumerate}[label=(\alph*), leftmargin=*]
        \item Let us show that for any $\Psi, \Phi \in \G(\KX)$, we have
        \[
            \aut^{\V, (0)}_{\Psi \circledast \Phi} = \aut^{\V, (0)}_{\Psi} \circ \aut^{\V, (0)}_{\Phi}.
        \]
        It suffices to prove this identity on generators. Since for $\Psi \in \G(\KX)$ we have $\aut^{\V, (0)}_{\Psi}(e_0)=e_0$ and $\aut^{\V, (0)}_{\Psi}(g)=g$, this is immediately true for $e_0$ and $g \in G$. Next,
        \[\begin{aligned}
            \aut^{\V, (0)}_{\Psi \circledast \Phi}(e_1) = & \beta((\Psi \circledast \Phi)^{-1} \otimes 1) e_1 \beta((\Psi \circledast \Phi) \otimes 1) \\
            = & \beta(\aut_{\Psi}(\Phi^{-1}) \Psi^{-1} \otimes 1) e_1 \beta(\Psi \aut_{\Psi}(\Phi) \otimes 1) \\
            = & \beta(\aut_{\Psi}(\Phi^{-1}) \otimes 1) \beta(\Psi^{-1} \otimes 1) e_1 \beta(\Psi \otimes 1) \beta(\aut_{\Psi}(\Phi) \otimes 1) \\
            = & \aut^{\V, (0)}_{\Psi}(\beta(\Phi^{-1} \otimes 1)) \aut^{\V, (0)}_{\Psi}(e_1) \aut^{\V, (0)}_{\Psi}(\beta(\Phi \otimes 1)) \\
            = & \aut^{\V, (0)}_{\Psi}(\beta(\Phi^{-1} \otimes 1) e_1 \beta(\Phi \otimes 1)) \\
            = & \aut^{\V, (0)}_{\Psi} \circ \aut^{\V, (0)}_{\Phi}(e_1)
        \end{aligned}\]
        where the fourth equality comes from the commutativity of Diagram (\ref{beta_auts}).
        \item Using Identity \ref{act_aut0}, we get
        \[\begin{aligned}
            \aut^{\V, (1)}_{\Psi} \circ \aut^{\V, (1)}_{\Phi} = & \Ad_{\beta(\Psi \otimes 1)} \circ \aut^{\V, (0)}_{\Psi} \circ \Ad_{\beta(\Phi \otimes 1)} \circ \aut^{\V, (0)}_{\Phi} \\
            = & \Ad_{\beta(\Psi \otimes 1)} \circ \Ad_{\aut^{\V, (0)}_{\Psi}(\beta(\Phi \otimes 1))} \circ \aut^{\V, (0)}_{\Psi} \circ \aut^{\V, (0)}_{\Phi} \\
            = & \Ad_{\beta(\Psi \otimes 1)\beta(\aut_{\Psi}(\Phi) \otimes 1)} \circ \aut^{\V, (0)}_{\Psi} \circ \aut^{\V, (0)}_{\Phi} \\
            = & \Ad_{\beta((\Psi \circledast \Phi) \otimes 1)} \circ \aut^{\V, (0)}_{\Psi \circledast \Phi} = \aut^{\V, (1)}_{\Psi \circledast \Phi},
        \end{aligned}\]
    \end{enumerate}

\end{proof}

\begin{propdef}
    \label{group_morph_W}
    For $\Psi \in \G(\KX)$, the automorphism $\aut^{\V, (1)}_{\Psi} : \hat{\V}_G \to \hat{\V}_G$ restricts to a topological $\K$-algebra automorphism on the $\K$-subalgebra $\hat{\W}_G$ which will be denoted $\aut^{\W, (1)}_{\Psi}$.
    Moreover, there is a group action of $(\G(\KX), \circledast)$ on $\hat{\W}_G$ by $\K$-algebra automorphisms
    \begin{equation}
        (\G(\KX), \circledast) \longrightarrow \Aut_{\K-\alg}(\hat{\W}_G), \, \Psi \longmapsto \aut^{\W, (1)}_{\Psi}.
    \end{equation}
\end{propdef}
\begin{proof}
    For $w=\lambda + v e_1 \in \hat{\W}_G$, we have
    \[\begin{aligned}
        \aut^{\V, (1)}_{\Psi}(w) = & \lambda + \aut^{\V, (1)}_{\Psi}(v) \beta(\Psi \otimes 1) \beta(\Psi^{-1} \otimes 1) e_1 \beta(\Psi \otimes 1) \beta(\Psi^{-1} \otimes 1) \\
        = & \lambda + \aut^{\V, (1)}_{\Psi}(v) e_1 \in \hat{\W}_G. 
    \end{aligned}\]
    This implies that $\aut^{\V, (1)}_{\Psi}$ induces a $\K$-algebra endomorphism of $\hat{\W}_G$. Moreover, the preimage of this endomorphism under the $\K$-module isomorphism $\K \times \hat{\V}_G \to \hat{\W}_G$, $(\lambda, v) \mapsto \lambda + v e_1$ is the pair $(\mathrm{id}, \aut_{\Psi}^{\V,(1)})$, which is a $\K$-module automorphism. This implies that $\aut_{\Psi}^{\W,(1)}$ is a $\K$-module automorphism, and therefore a $\K$-algebra automorphism. Thanks to this, the second part of this result can be deduced from Proposition \ref{group_morph_V}\ref{act_aut1}, by restriction on $\hat{\W}_G$. 
\end{proof}

\subsubsection{Action of \texorpdfstring{$(\G(\KX), \circledast)$}{GKX*} by module automorphisms}
\begin{definition}
    For $\Psi \in \G(\KX)$, we define $\aut^{\V, (10)}_{\Psi}$ to be the topological $\K$-module automorphism of $\hat{\V}_G$ given by
    \begin{equation}
        \label{eq_aut10}
        \aut^{\V, (10)}_{\Psi} := \ell_{\beta(\Psi \otimes 1)} \circ \aut^{\V, (0)}_{\Psi},
    \end{equation}
\end{definition}
\noindent \textbf{Remark}. Let us notice that, for any $\Psi \in \G(\KX)$, we also have
\[
    \aut^{\V, (10)}_{\Psi} = \ell_{\beta(\Psi \otimes 1)} \circ \aut^{\V, (0)}_{\Psi} = \ell_{\beta(\Psi \otimes 1)} \circ \Ad_{\beta(\Psi^{-1} \otimes 1)} \circ \aut^{\V, (1)}_{\Psi} = r_{\beta(\Psi \otimes 1)} \circ \aut^{\V, (1)}_{\Psi}.
\]
\begin{proposition}
    \label{act_aut10}
    There is a group action of $(\G(\KX), \circledast)$ on $\hat{\V}_G$ by $\K$-module automorphisms
        \[
            (\G(\KX), \circledast) \longrightarrow \Aut_{\K-\Mod}(\hat{\V}_G), \, \Psi \longmapsto \aut^{\V, (10)}_{\Psi}
        \]
\end{proposition}
\begin{proof}
        For $\Psi, \Phi \in \G(\KX)$, we have
        \[\begin{aligned}
            \aut^{\V, (10)}_{\Psi} \circ \aut^{\V, (10)}_{\Phi} = & \ell_{\beta(\Psi \otimes 1)} \circ \aut^{\V, (0)}_{\Psi} \circ \ell_{\beta(\Phi \otimes 1)} \circ \aut^{\V, (0)}_{\Phi} \\
            = & \ell_{\beta(\Psi \otimes 1)} \circ \ell_{\aut^{\V, (0)}_{\Psi}(\beta(\Phi \otimes 1))} \circ \aut^{\V, (0)}_{\Psi} \circ \aut^{\V, (0)}_{\Phi} \\
            =& \ell_{\beta(\Psi \otimes 1)\beta(\aut_{\Psi}(\Phi) \otimes 1)} \circ \aut^{\V, (0)}_{\Psi} \circ \aut^{\V, (0)}_{\Phi} \\
            = & \ell_{\beta((\Psi \circledast \Phi) \otimes 1)} \circ \aut^{\V, (0)}_{\Psi \circledast \Phi} = \aut^{\V, (10)}_{\Psi \circledast \Phi}.
        \end{aligned}\]
        where the last equality comes from the commutativity of Diagram (\ref{beta_auts}) and from Proposition \ref{group_morph_V}\ref{act_aut0}. 
\end{proof}

\begin{lemma}
    \label{rel_aut_alg_mod}
    For any $\Psi \in \G(\KX)$, we have the following identities:
    \begin{align}
        & \forall a, b \in \hat{\V}_G, \, \aut^{\V, (10)}_{\Psi}(ab) = \aut^{\V, (10)}_{\Psi}(a) \, \aut^{\V, (0)}_{\Psi}(b). \label{rel_aut_100} & \\
        & \forall a, b \in \hat{\V}_G, \, \aut^{\V, (10)}_{\Psi}(ab) = \aut^{\V, (1)}_{\Psi}(a) \, \aut^{\V, (10)}_{\Psi}(b). \label{rel_aut_101} &
    \end{align}
\end{lemma}
\begin{proof}
    Let $a, b \in \hat{\V}_G$. We have
    \begin{align*}
        \aut^{\V, (10)}_{\Psi}(ab) = & \ell_{\beta(\Psi \otimes 1)} \circ \aut^{\V, (0)}_{\Psi}(ab) = \ell_{\beta(\Psi \otimes 1)} \left( \aut^{\V, (0)}_{\Psi}(a) \, \aut^{\V, (0)}_{\Psi}(b) \right) \\
        = & \left(\ell_{\beta(\Psi \otimes 1)} \circ \aut^{\V, (0)}_{\Psi}(a)\right) \aut^{\V, (0)}_{\Psi}(b) = \aut^{\V, (10)}_{\Psi}(a) \, \aut^{\V, (0)}_{\Psi}(b)   
    \end{align*}
    and
    \begin{align*}
        \aut^{\V, (10)}_{\Psi}(ab) = & r_{\beta(\Psi \otimes 1)} \circ \aut^{\V, (1)}_{\Psi}(ab) = r_{\beta(\Psi \otimes 1)} \left( \aut^{\V, (1)}_{\Psi}(a) \, \aut^{\V, (1)}_{\Psi}(b) \right) \\
        = & \aut^{\V, (1)}_{\Psi}(a) \left(r_{\beta(\Psi \otimes 1)} \circ \aut^{\V, (1)}_{\Psi}(b)\right) = \aut^{\V, (1)}_{\Psi}(a) \, \aut^{\V, (10)}_{\Psi}(b)    
    \end{align*}
\end{proof}

\begin{proposition}
    \label{aut_V10_preserves}
    For $\Psi \in \G(\KX)$, the $\K$-module automorphism $\aut^{\V, (10)}_{\Psi}$ preserves the submodule $\hat{\V}_G e_0 + \sum_{g \in G \, \backslash \, \{1\}} \hat{\V}_G (g-1)$.
\end{proposition}

\begin{proof}
    Using Lemma \ref{rel_aut_alg_mod}\ref{rel_aut_100}, we obtain for any $a \in \hat{\V}_G$ and $(b_g)_{g \in G} \in (\hat{\V}_G)^G$ :
    \begin{align*}
        &\aut^{\V, (10)}_{\Psi}\left(ae_0 + \sum_{g \in G \backslash \{1\}} b_g(g-1)\right) = \aut^{\V, (10)}_{\Psi}(a) \aut^{\V, (0)}_{\Psi}(e_0) + \sum_{g \in G} \aut^{\V, (10)}_{\Psi}(b_g) \aut^{\V, (0)}_{\Psi}(g-1) \\
        &= \aut^{\V, (10)}_{\Psi}(a) e_0 + \sum_{g \in G \backslash \{1\}} \aut^{\V, (10)}_{\Psi}(b_g) (g-1) \in \hat{\V}_G e_0 + \sum_{g \in G \, \backslash \, \{1\}} \hat{\V}_G (g-1).  
    \end{align*}
\end{proof}

\begin{propdef}
    \label{aut_M10}
    For $\Psi \in \G(\KX)$, there is a unique $\K$-module automorphism $\aut^{\M, (10)}_{\Psi}$ of $\hat{\M}_G$ such that the following diagram
    \[
        \begin{tikzcd}
            \hat{\V}_G \ar[rrr, "\aut^{\V, (10)}_{\Psi}"] \ar[d, "- \cdot 1_{\M}"'] &&& \hat{\V}_G \ar[d, "- \cdot 1_{\M}"] \\
            \hat{\M}_G \ar[rrr, "\aut^{\M, (10)}_{\Psi}"'] &&& \hat{\M}_G
        \end{tikzcd}
    \]
    commutes.
\end{propdef}
\begin{proof}
    Follows from Proposition \ref{aut_V10_preserves}.
\end{proof}

\begin{lemma}
    \label{compat_MV}
    For any $\Psi \in \G(\KX)$, we have
    \begin{align}
        & \forall (a, m) \in \hat{\V}_G \times \hat{\M}_G, \, \aut^{\M, (10)}_{\Psi}(a \cdot m) = \aut^{\V, (1)}_{\Psi}(a) \cdot \aut^{\M, (10)}_{\Psi}(m). \\
        & \forall (w, m) \in \hat{\W}_G \times \hat{\M}_G, \, \aut^{\M, (10)}_{\Psi}(w \cdot m) = \aut^{\W, (1)}_{\Psi}(w) \cdot \aut^{\M, (10)}_{\Psi}(m).
    \end{align}
\end{lemma}
\begin{proof}
    The first identity is proved by using a combination of Proposition-Definition \ref{aut_M10} and Lemma \ref{rel_aut_alg_mod}(\ref{rel_aut_101}). The second identity can be deduced from the first by restriction on the subalgebra $\hat{\W}_G$.
\end{proof}

\begin{corollary}
    \label{group_morph_aut_M_10}
    There is a group action of $(\G(\KX), \circledast)$ on $\hat{\M}_G$ by topological $\K$-module automorphisms
    \begin{equation}
        (\G(\KX), \circledast) \longrightarrow \Aut_{\K-\Mod}\left(\hat{\M}_G \right), \, \Psi \longmapsto \aut^{\M, (10)}_{\Psi}
    \end{equation}
\end{corollary}
\begin{proof}
    It is a combination of Proposition-Definition \ref{aut_M10} and Proposition \ref{act_aut10}.  
\end{proof}

\subsection{The cocycle \texorpdfstring{$\Gamma$}{Gamma} and twisted actions} \label{gamma_actions_on_crossed_product}
To an element $\Psi \in \G(\KX)$, one associates $\Gamma_{\Psi} \in \K[[x]]^{\times}$ (see (\ref{Gamma_function})). Then $\Gamma_{\Psi}(-e_1)$ is an invertible element of $\hat{\V}_G$. 
\begin{definition}
    For $\Psi \in \G(\KX)$, we define the topological $\K$-algebra automorphism of $\hat{\V}_G$:
    \begin{equation}
        \label{Def_Gamma_aut_V_1}
        ^{\Gamma}\aut^{\V, (1)}_{\Psi} := \Ad_{\Gamma^{-1}_{\Psi}(-e_1)} \circ \aut^{\V, (1)}_{\Psi}.
    \end{equation}    
\end{definition}

\begin{propdef}
    \label{restrict_Gammaaut}
    For $\Psi \in \G(\KX)$, the automorphism $^{\Gamma}\aut^{\V, (1)}_{\Psi}$ restricts to a topological $\K$-algebra automorphism of the subalgebra $\hat{\W}_G$ denoted $^{\Gamma}\aut^{\W, (1)}_{\Psi}$.
\end{propdef}
\begin{proof}
    Follows from Proposition-Definition \ref{group_morph_W} and the fact that $\Gamma_{\Psi}(-e_1)$ is an invertible element of $\hat{\W}_G$.
\end{proof}

\begin{proposition} \ 
    \label{Gamma_aut_VW_1}
    \begin{enumerate}[label=(\alph*), leftmargin=*]
        \item \label{Gamma_aut_V_1} There is a group action of $(\G(\KX), \circledast)$ on $\hat{\V}_G$ by topological $\K$-algebra automorphisms
        \begin{equation}
            (\G(\KX), \circledast) \longrightarrow  \Aut_{\K-\Mod}\left(\hat{\V}_G\right), \,\, \Psi \longmapsto \,^{\Gamma}\aut^{\V, (1)}_{\Psi}
        \end{equation}
        \item \label{Gamma_aut_W_1} There is a group action of $(\G(\KX), \circledast)$ on $\hat{\W}_G$ by topological $\K$-module automorphisms
        \begin{equation}
            (\G(\KX), \circledast) \longrightarrow  \Aut_{\K-\Mod}\left(\hat{\W}_G\right), \,\, \Psi \longmapsto \,^{\Gamma}\aut^{\W, (1)}_{\Psi}
        \end{equation}
    \end{enumerate}
\end{proposition}
\begin{proof} \ 
    \begin{enumerate}[label=(\alph*), leftmargin=*]
        \item It follows from Proposition \ref{group_morph_V}\ref{act_aut1}, Lemma \ref{Gamma_aut} and the fact that $e_1$ is invariant under $\aut_{\Psi}^{\V, (1)}$ for any $\Psi \in \G(\KX)$.
        \item It follows from \ref{Gamma_aut_V_1} thanks to Proposition-Definition \ref{group_morph_W}.
    \end{enumerate}
\end{proof}

\begin{definition}
    For $\Psi \in \G(\KX)$, we define the following topological $\K$-module automorphism of $\hat{\M}_G$:
    \begin{equation}
        \label{Def_Gamma_aut_M_10}
        ^{\Gamma}\aut^{\M, (10)}_{\Psi} := \ell_{\Gamma^{-1}_{\Psi}(-e_1)} \circ \aut^{\M, (10)}_{\Psi}.
    \end{equation}
\end{definition}

\begin{lemma}
    \label{compat_Gamma_MV}
    For any $\Psi \in \G(\KX)$, we have
    \begin{align}
        &\forall (a, m) \in \hat{\V}_G \times \hat{\M}_G, \, ^{\Gamma}\aut^{\M, (10)}_{\Psi}(a \cdot m) = \, ^{\Gamma}\aut^{\V, (1)}_{\Psi}(a) \cdot \, ^{\Gamma}\aut^{\M, (10)}_{\Psi}(m) \\
        &\forall (w, m) \in \hat{\W}_G \times \hat{\M}_G, \, ^{\Gamma}\aut^{\M, (10)}_{\Psi}(w \cdot m) = \, ^{\Gamma}\aut^{\W, (1)}_{\Psi}(w) \cdot \, ^{\Gamma}\aut^{\M, (10)}_{\Psi}(m)
    \end{align}
\end{lemma}
\begin{proof}
    Follows by a computation involving Lemma \ref{compat_MV}.
\end{proof}

\begin{proposition}
    \label{Gamma_aut_M_10}
    There is a group action of $(\G(\KX), \circledast)$ on $\hat{\M}_G$ by topological $\K$-module automorphisms
    \begin{equation}
        (\G(\KX), \circledast) \longrightarrow  \Aut_{\K-\Mod}\left(\hat{\M}_G\right), \,\, \Psi \longmapsto \,^{\Gamma}\aut^{\M, (10)}_{\Psi}
    \end{equation}
\end{proposition}
\begin{proof}
    For any $\Psi, \Phi \in \G(\KX)$, we have
    \begin{align*}
        ^{\Gamma}\aut^{\M, (10)}_{\Psi \circledast \Phi} = & \ell_{\Gamma^{-1}_{\Psi \circledast \Phi}(-e_1)} \circ \aut^{\M, (10)}_{\Psi \circledast \Phi} = \ell_{\Gamma^{-1}_{\Psi}(-e_1) \Gamma^{-1}_{\Phi}(-e_1)} \circ \aut^{\M, (10)}_{\Psi} \circ \aut^{\M, (10)}_{\Phi} \\
        = & \ell_{\Gamma^{-1}_{\Psi}(-e_1)} \circ \ell_{\Gamma^{-1}_{\Phi}(-e_1)} \circ \aut^{\M, (10)}_{\Psi} \circ \aut^{\M, (10)}_{\Phi} \\
        = & \ell_{\Gamma^{-1}_{\Psi}(-e_1)} \circ \aut^{\M, (10)}_{\Psi} \circ \ell_{\Gamma^{-1}_{\Phi}(-e_1)} \circ \aut^{\M, (10)}_{\Phi} \\
        = & ^{\Gamma}\aut^{\M, (10)}_{\Psi} \circ \, ^{\Gamma}\aut^{\M, (10)}_{\Phi} 
    \end{align*}
    where the seoncd equality uses Lemma \ref{Gamma_aut} and Corollary \ref{group_morph_aut_M_10}; and the fourth equality comes from the following computation:
    \begin{align*}
        \ell_{\Gamma^{-1}_{\Phi}(-e_1)} \circ \aut_{\Psi}^{\M, (10)}(m) = & \Gamma^{-1}_{\Phi}(-e_1) \, \aut_{\Psi}^{\M, (10)}(m) = \aut_{\Psi}^{\V, (1)}(\Gamma^{-1}_{\Phi}(-e_1)) \, \aut_{\Psi}^{\M,(10)}(m) \\
        = & \aut_{\Psi}^{\M,(10)} (\Gamma^{-1}_{\Phi}(-e_1)m) = \aut_{\Psi}^{\M, (10)} \circ \ell_{\Gamma^{-1}_{\Phi}(-e_1)}(m) 
    \end{align*}
    for any $m \in \hat{\M}_G$; where the second equality uses the fact $e_1$ is invariant under $\aut_{\Psi}^{\V, (1)}$ and the third equality comes from Lemma \ref{compat_MV}.  
\end{proof}

\subsection{The stabilizer groups \texorpdfstring{$\Stab(\hat{\Delta}^{\W}_{G})(\K)$}{StabW} and \texorpdfstring{$\Stab(\hat{\Delta}^{\M}_{G})(\K)$}{StabM}} \label{Stabs}
Using Proposition \ref{Gamma_aut_VW_1}, we define the following group action  of $(\G(\KX), \circledast)$ on $\Mor_{\K-\alg}\left(\hat{\W}_G, \left(\hat{\W}_G\right)^{\hat{\otimes} 2}\right)$:
\begin{equation}
    \label{act_on_DeltaW}
    \Psi \cdot D^{\W} := \left({^{\Gamma}\aut^{\W, (1)}_{\Psi}}\right)^{\otimes 2} \circ D^{\W} \circ \left(^{\Gamma}\aut^{\W, (1)}_{\Psi}\right)^{-1}.
\end{equation}
In particular, the stabilizer of $\hat{\Delta}^{\W}_{G}$ is the subgroup
\begin{equation}
    \label{Stab_DeltaW}
    \Stab(\hat{\Delta}^{\W}_{G})(\K) := \left\{ \Psi \in \G(\KX) \, | \, \left( {^{\Gamma}\aut^{\W, (1)}_{\Psi}} \right)^{\otimes 2} \circ \hat{\Delta}^{\W}_{G} = \hat{\Delta}^{\W}_{G} \circ {^{\Gamma}\aut^{\W, (1)}_{\Psi}} \right\}.
\end{equation}

Similarly, Proposition \ref{Gamma_aut_M_10} provides a group action of $(\G(\KX), \circledast)$ on the $\K$-module $\Mor_{\K-\Mod}\left(\hat{\M}_G, \left(\hat{\M}_G\right)^{\hat{\otimes} 2}\right)$:
\begin{equation}
    \label{act_on_DeltaM}
    \Psi \cdot D^{\M} := \left({^{\Gamma}\aut^{\M, (10)}_{\Psi}}\right)^{\otimes 2} \circ D^{\M} \circ \left(^{\Gamma}\aut^{\M, (10)}_{\Psi}\right)^{-1}.
\end{equation}
In particular, the stabilizer of $\hat{\Delta}^{\M}_{G}$ is the subgroup
\begin{equation}
    \label{Stab_DeltaM}
    \Stab(\hat{\Delta}^{\M}_{G})(\K) := \left\{ \Psi \in \G(\KX) \, | \, \left( {^{\Gamma}\aut^{\M, (10)}_{\Psi}} \right)^{\otimes 2} \circ \hat{\Delta}^{\M}_{G} = \hat{\Delta}^{\M}_{G} \circ {^{\Gamma}\aut^{\M, (10)}_{\Psi}} \right\}.
\end{equation}

\noindent We then have the following inclusion of subgroups
\begin{theorem}
    \label{Stab_Inclusion}
    $\Stab(\hat{\Delta}^{\M}_{G})(\K) \subset \Stab(\hat{\Delta}^{\W}_{G})(\K)$ (as subgroups of $(\G(\KX), \circledast)$).
\end{theorem}
\begin{proof}
    Let $\Psi \in \Stab(\hat{\Delta}^{\M}_{G})(\K)$. First, let us notice that
    \begin{align}
        & \left(\Gamma^{-1}_{\Psi}(-e_1) \beta(\Psi \otimes 1) \cdot 1_{\M}\right)^{\otimes 2} = \left(\,^{\Gamma}\aut^{\M, (10)}_{\Psi}(1_{\M})\right)^{\otimes 2} \label{Image_of_1} \\
        & = \left(\,^{\Gamma}\aut^{\M, (10)}_{\Psi}\right)^{\otimes 2} \circ \hat{\Delta}^{\M}_G(1_{\M}) = \hat{\Delta}^{\M}_G \circ \, ^{\Gamma}\aut^{\M, (10)}_{\Psi}(1_{\M}), \notag
    \end{align}
    where the last equality follows from the assumption on $\Psi$. Then for any $w \in \hat{\W}_G$,
    \begin{align}
        \label{mid_computation}
        &\hat{\Delta}^{\W}_G\left(\,^{\Gamma}\aut^{\W, (1)}_{\Psi}(w)\right) \cdot \left(\Gamma^{-1}_{\Psi}(-e_1) \beta(\Psi \otimes 1) \cdot 1_{\M}\right)^{\otimes 2} \\
        &= \hat{\Delta}^{\W}_G\left(\,^{\Gamma}\aut^{\W, (1)}_{\Psi}(w)\right) \cdot \hat{\Delta}^{\M}_G \left(\,^{\Gamma}\aut^{\M, (10)}_{\Psi}(1_{\M})\right) \notag \\ 
        &= \hat{\Delta}^{\M}_G \left(\,^{\Gamma}\aut^{\W, (1)}_{\Psi}(w) \cdot ^{\Gamma}\aut^{\M, (10)}_{\Psi}(1_{\M}) \right) = \hat{\Delta}^{\M}_G \left(\, ^{\Gamma}\aut^{\M, (10)}_{\Psi}(w \cdot 1_{\M}) \right) \notag \\
        &= \left(\, ^{\Gamma}\aut^{\M, (10)}_{\Psi} \right)^{\otimes 2} \circ \hat{\Delta}^{\M}_G (w \cdot 1_{\M}) = \left(\, ^{\Gamma}\aut^{\M, (10)}_{\Psi} \right)^{\otimes 2} \left(\hat{\Delta}^{\W}_G (w) \cdot \hat{\Delta}^{\M}_G(1_{\M})\right) \notag \\
        &= \left(\, ^{\Gamma}\aut^{\W, (1)}_{\Psi} \right)^{\otimes 2} \left(\hat{\Delta}^{\W}_G (w)\right) \cdot \left(\, ^{\Gamma}\aut^{\M, (10)}_{\Psi} \right)^{\otimes 2} \left(\hat{\Delta}^{\M}_G(1_{\M})\right) \notag \\
        &= \left(\, ^{\Gamma}\aut^{\W, (1)}_{\Psi} \right)^{\otimes 2} \left(\hat{\Delta}^{\W}_G (w)\right) \cdot \left(\Gamma^{-1}_{\Psi}(-e_1) \beta(\Psi \otimes 1) \cdot 1_{\M}\right)^{\otimes 2} \notag
    \end{align}
    where the first and seventh equality come from (\ref{Image_of_1}), the second and the fifth from Proposition-Definition \ref{DeltaW_DeltaM}\ref{compat_DeltaW_DeltaM}, the third and the sixth from Lemma \ref{compat_Gamma_MV} and the fourth from the fact that $\Psi \in \Stab(\hat{\Delta}^{\M}_G)(\K)$. Next, since $\Gamma^{-1}_{\Psi}(-e_1) \beta(\Psi \otimes 1)$ is invertible in $\hat{\V}_G$, the map $\hat{\W}_G \to \hat{\M}_G, \, w \mapsto w \, \Gamma^{-1}_{\Psi}(-e_1) \beta(\Psi \otimes 1) \cdot 1_{\M}$ is an isomorphism of left $\hat{\W}_G$-modules. Consequently, Identity (\ref{mid_computation}) implies that
    \begin{equation}
        \forall w \in \hat{\W}_G, \, \left(\, ^{\Gamma}\aut^{\W, (1)}_{\Psi} \right)^{\otimes 2} \left(\hat{\Delta}^{\W}_G (w)\right) = \hat{\Delta}^{\W}_G\left(\,^{\Gamma}\aut^{\W, (1)}_{\Psi}(w)\right),  
    \end{equation}
    thus establishing that $\Psi \in \Stab(\hat{\Delta}^{\W}_G)(\K)$.
\end{proof}
	\section{The stabilizer groups in terms of Racinet's formalism} \label{StabRac}
In this part, we translate the inclusion of stabilizers in Theorem \ref{Stab_Inclusion} into Racinet's formalism. In §\ref{sect_StabM_Stab*}, we relate the various $(\G(\KX), \circledast)$-actions we recalled from \cite{Rac} in §\ref{Racinet_formalism} and the ones we constructed in §\ref{crossed_product}.
This allows us to identify the group $\Stab(\hat{\Delta}^{\M}_G)$ from (\ref{Stab_DeltaM}) with the group $\Stab(\hat{\Delta}^{\Mod}_{\star})$ from \cite{EF0}. In §\ref{sect_StabW_Stab*}, we transport the action of the group $(\G(\KX), \circledast)$ on $\hat{\W}_G$ given in \ref{Gamma_aut_VW_1}\ref{Gamma_aut_W_1} into an action of the same group on the algebra $\KY$ and express the latter action in terms of Racinet's formalism. This enables us to identify the stabilizer group $\Stab(\hat{\Delta}^{\W}_G)$ given in \ref{Stab_DeltaW} with a group $\Stab(\hat{\Delta}^{\alg}_{\star})$ defined in the framework of Racinet's formalism.
The inclusion of stabilizers from Theorem \ref{Stab_Inclusion} is then expressed as the inclusion $\Stab(\hat{\Delta}^{\Mod}_{\star}) \subset \Stab(\hat{\Delta}^{\alg}_{\star})$ (see Corollary \ref{StabDeltaAlg_StabDeltaMod}).

\subsection{Identification of the subgroups \texorpdfstring{$\Stab(\hat{\Delta}^{\M}_G)$}{StabDeltaM} and \texorpdfstring{$\Stab(\hat{\Delta}^{\Mod}_{\star})$}{StabDelta*}} \label{sect_StabM_Stab*}
\subsubsection{A \texorpdfstring{$(\G(\KX), \circledast)$}{GKX}-module isomorphism}
Let us recall $\beta : \KX \rtimes G \to \hat{\V}_G$ the $\K$-algebra isomorphism given in Proposition \ref{isoVGetKXG} \ref{KXGtoVG}.
\begin{lemma}
    For $\Psi \in \G(\KX)$, the following diagram
    \begin{equation}
        \label{diag_link_aut_aut0}
        \begin{tikzcd}
            \KX \ar["\beta \circ (- \otimes 1)"]{rr} \ar["\aut_{\Psi}"']{d} && \hat{\V}_G \ar["\aut^{\V, (0)}_{\Psi}"]{d} \\
            \KX \ar["\beta \circ (- \otimes 1)"']{rr} && \hat{\V}_G
        \end{tikzcd}
    \end{equation}
    commutes; where $\aut_{\Psi}$ is the $\K$-algebra automorphism of $\KX$ given in (\ref{aut_Psi}) and $\aut^{\V, (0)}_{\Psi}$ is the $\K$-algebra automorphism of $\hat{\V}_G$ given in Proposition-Definition \ref{aut0}.
    \label{link_aut_aut0}
\end{lemma}
\begin{proof}
    This is done by left composing Diagram (\ref{beta_auts}) with the following commutative diagram
    \begin{equation*}
        \begin{tikzcd}
            \KX \ar["- \otimes 1"]{r} \ar["\aut_{\Psi}"']{d} & \KX \rtimes G \ar["\aut_{\Psi} \otimes \mathrm{id}_{\K G}"]{d} \\
            \KX \ar["- \otimes 1"']{r} & \KX \rtimes G
        \end{tikzcd}
    \end{equation*}
\end{proof}

\begin{lemma}
    For $\Psi \in \G(\KX)$, the following diagram
    \begin{equation}
        \label{diag_link_S_aut10}
        \begin{tikzcd}
            \KX \ar["\beta \circ (- \otimes 1)"]{rr} \ar["S_{\Psi}"']{d} && \hat{\V}_G \ar["\aut^{\V, (10)}_{\Psi}"]{d} \\
            \KX \ar["\beta \circ (- \otimes 1)"']{rr} && \hat{\V}_G
        \end{tikzcd}
    \end{equation}
    commutes; where $S_{\Psi}$ is the $\K$-module automorphism of $\KX$ given in (\ref{eq_S_Psi}) and $\aut^{\V, (10)}_{\Psi}$ is the $\K$-module automorphism of $\hat{\V}_G$ given in (\ref{eq_aut10}).
    \label{link_S_aut10}
\end{lemma}
\begin{proof}
    Thanks to Identities (\ref{eq_S_Psi}) and (\ref{eq_aut10}), this is done by composing from the bottom Diagram (\ref{diag_link_aut_aut0}) with the following commutative diagram
    \begin{equation*}
        \begin{tikzcd}
            \KX \ar["\beta \circ (- \otimes 1)"]{rr} \ar["\ell_{\Psi}"']{d} && \hat{\V}_G \ar["\ell_{\beta(\Psi \otimes 1)}"]{d} \\
            \KX \ar["\beta \circ (- \otimes 1)"']{rr} && \hat{\V}_G
        \end{tikzcd}
    \end{equation*} 
\end{proof}

\begin{lemma}
    For $\Psi \in \G(\KX)$, the following diagram
    \begin{equation}
        \label{diag_link_SY_autM10}
        \begin{tikzcd}
            \KX/\KX x_0 \ar["\kappa \circ \overline{\q}^{-1}"]{rr} \ar["S_{\Psi}^Y"']{d} && \hat{\M}_G \ar["\aut^{\M, (10)}_{\Psi}"]{d} \\
            \KX/\KX x_0 \ar["\kappa \circ \overline{\q}^{-1}"']{rr} && \hat{\M}_G
        \end{tikzcd}
    \end{equation}
    commutes; where $S^Y_{\Psi}$ is the $\K$-module automorphism of $\KX / \KX x_0$ given in Proposition-Definition \ref{S_Y}, $\aut^{\M, (10)}_{\Psi}$ is the $\K$-module automorphism of $\hat{\M}_G$ given in Proposition-Definition \ref{aut_M10}, $\kappa : \KX / \KX x_0 \to \hat{\M}_G$ is the $\K$-module isomorphism in Proposition \ref{isoMG} and $\overline{\q}$ is the $\K$-module automorphism of $\KX / \KX x_0$ given in §\ref{basic_objects}. 
    \label{link_SY_autM10}
\end{lemma}
\begin{proof}
    Let us consider the following cube
    \begin{equation*}
        \begin{tikzcd}[row sep={40,between origins}, column sep={40,between origins}]
            &  & \mathbf{k}\langle\langle X \rangle\rangle/\mathbf{k}\langle\langle X \rangle\rangle x_0 \arrow[ddd, "S^Y_{\Psi}"] \arrow[rrrr, "\kappa \circ \overline{\q}^{-1}"] & & & & \hat{\mathcal{M}}_G \arrow[ddd, "{\mathrm{aut}^{\mathcal{M}, (10)}_{\Psi}}"] \\
            \mathbf{k}\langle\langle X \rangle\rangle \arrow[rrrr, "\hspace{1cm}\beta \circ (- \otimes 1)"] \arrow[ddd, "S_{\Psi}"'] \arrow[rru, "\overline{\q} \circ \pi_Y"] & & & & \hat{\mathcal{V}}_G \arrow[ddd, "{\mathrm{aut}^{\mathcal{V}, (10)}_{\Psi}}"] \arrow[rru, "- \cdot 1_{\mathcal{M}}"] & & \\
            & & & & & & \\
            & & \mathbf{k}\langle\langle X \rangle\rangle/\mathbf{k}\langle\langle X \rangle\rangle x_0 \arrow[rrrr, "\kappa \circ \overline{\q}^{-1}"'] & & & & \hat{\mathcal{M}}_G \\
            \mathbf{k}\langle\langle X \rangle\rangle \arrow[rrrr, "\beta \circ (- \otimes 1)"'] \arrow[rru, "\overline{\q} \circ \pi_Y"] & & & & \hat{\mathcal{V}}_G \arrow[rru, "- \cdot 1_{\mathcal{M}}"'] & &
        \end{tikzcd}
    \end{equation*}
    First, the left (resp. right) side commutes by definition of $S_{\Psi}^Y$ (resp. $\aut_{\Psi}^{\M, (10)}$). Then, the upper and lower sides are exactly the same square, which is commutative thanks to Proposition \ref{isoMG}. Finally, Lemma \ref{link_S_aut10} gives us the commutativity of the front side. This collection of commutativities together with the surjectivity of $\overline{\q} \circ \pi_Y$ implies that the back side of the cube commutes, which is exactly Diagram (\ref{diag_link_SY_autM10}).
\end{proof}

\begin{proposition}
    For $\Psi \in \G(\KX)$, the following diagram
    \begin{equation}
        \label{diag_link_GammaSY_GammaautM10}
        \begin{tikzcd}
            \KX/\KX x_0 \ar["\kappa \circ \overline{\q}^{-1}"]{rr} \ar["^{\Gamma}S_{\Psi}^Y"']{d} && \hat{\M}_G \ar["^{\Gamma}\aut^{\M, (10)}_{\Psi}"]{d} \\
            \KX/\KX x_0 \ar["\kappa \circ \overline{\q}^{-1}"']{rr} && \hat{\M}_G
        \end{tikzcd}
    \end{equation}
    commutes; where $^{\Gamma}S_{\Psi}^Y$ and $^{\Gamma}\aut^{\M, (10)}_{\Psi}$ are respectively $\K$-module automorphisms of $\KX/\KX x_0$ and $\hat{\M}_G$. It follows that $\kappa \circ \overline{\q}^{-1}$ is an isomorphism between the $(\G(\KX), \circledast)$-modules\footnote{see Corollary \ref{Gamma_SY} (resp. Proposition \ref{Gamma_aut_M_10}) for the $(\G(\KX), \circledast)$-module structure of $\KX/\KX x_0$ (resp. $\hat{\M}_G$)} $\KX/\KX x_0$ and $\hat{\M}_G$.
    \label{link_GammaSY_GammaautM10}
\end{proposition}
\begin{proof}
    This is done by composing from the bottom Diagram (\ref{diag_link_SY_autM10}) with the following diagram
    \begin{equation*}
        \begin{tikzcd}
            \KX/\KX x_0 \ar["\kappa \circ \overline{\q}^{-1}"]{rr} \ar["\ell_{\Gamma_{\Psi}^{-1}(x_1)}"']{d} && \hat{\M}_G \ar["\ell_{\Gamma_{\Psi}^{-1}(-e_1)}"]{d} \\
            \KX/\KX x_0 \ar["\kappa \circ \overline{\q}^{-1}"']{rr} && \hat{\M}_G
        \end{tikzcd}
    \end{equation*}
    which we show is commutative. Indeed, 
    \begin{align*}
        & \ell_{\Gamma_{\Psi}^{-1}(-e_1)} \circ \kappa \circ \overline{\q}^{-1} \circ \overline{\q} \circ \pi_Y = \ell_{\Gamma_{\Psi}^{-1}(-e_1)} \circ (- \cdot 1_{\M}) \circ \beta \circ (- \otimes 1) \\
        &= (- \cdot 1_{\M}) \circ \ell_{\beta(\Gamma_{\Psi}^{-1}(x_1) \otimes 1)} \circ \beta \circ (- \otimes 1) = (- \cdot 1_{\M}) \circ \beta \circ (- \otimes 1) \circ \ell_{\Gamma_{\Psi}^{-1}(x_1)} \\
        & = \kappa \circ \overline{\q}^{-1} \circ \overline{\q} \circ \pi_Y \circ \ell_{\Gamma_{\Psi}^{-1}(x_1)}
        = \kappa \circ \overline{\q}^{-1} \circ \ell_{\Gamma_{\Psi}^{-1}(x_1)} \circ \overline{\q} \circ \pi_Y 
    \end{align*}
    where the first and fourth equalities come from the commutativity of Diagram (\ref{diag_iso_MG}); the second one from the fact that $- \cdot 1_{\M} : \hat{\V}_G \to \hat{\M}_G$ is a $\hat{\V}_G$-module morphism; the third one from the fact that $\beta \circ (- \otimes 1) : \KX \to \hat{\V}_G$ is a $\K$-algebra morphism and the last one from the fact that $\pi_Y : \KX \to \KX / \KX x_0$ is $\KX$-module morphism and that for any $a \in \KX$, $\q(x_1 a) = x_1 \q(a)$. \newline
    Finally, since $\overline{\q} \circ \pi_Y$ is a surjective $\K$-module morphism, it follows that
    \[
        \ell_{\Gamma_{\Psi}^{-1}(-e_1)} \circ \kappa \circ \overline{\q}^{-1} = \kappa \circ \overline{\q}^{-1} \circ \ell_{\Gamma_{\Psi}^{-1}(x_1)}, 
    \]
    which is the wanted result.
\end{proof}

\subsubsection{An isomorphism of coalgebras}
Let us recall $\varpi : \KY \to \hat{\W}_G$ the $\K$-algebra isomorphism given in Corollary \ref{iso_WG_MG}\ref{item_diag_projections}.
\begin{lemma}
    The following diagram
    \begin{equation}
        \label{diag_link_DeltaW_Delta*}
        \begin{tikzcd}
            \KY \ar["\varpi"]{rr} \ar["\hat{\Delta}^{\alg}_{\star}"']{d} && \hat{\W}_G \ar["\hat{\Delta}^{\W}_G"]{d} \\
            \KY^{\otimes 2} \ar["\varpi^{\otimes 2}"']{rr} && \hat{\W}_G^{\otimes 2}
        \end{tikzcd}
    \end{equation}
    commutes; where $\hat{\Delta}^{\alg}_{\star}$ (resp. $\hat{\Delta}^{\W}_G$) is the coproduct of $\KY$ (resp. $\hat{\W}_G$) given in (\ref{harmonic_coproduct}) (resp. (\ref{DeltaW})). In other words, the map $\varpi$ is a bialgebra isomorphism.
    \label{link_DeltaW_Delta*}
\end{lemma}
\begin{proof}
    Since all arrows on diagram (\ref{diag_link_DeltaW_Delta*}) are $\K$-algebra morphisms, it is enough to work on generators. For $(n, g) \in \N^{\ast} \times G$ we have
    \begin{align*}
        \varpi^{\otimes 2} \circ \hat{\Delta}^{\alg}_{\star}(y_{n, g}) = & \varpi^{\otimes 2}\left( y_{n,g} \otimes 1 + 1 \otimes y_{n,g} + \sum_{\substack{k=1 \\ h \in G}}^{n-1} y_{k,h} \otimes y_{n-k,hg^{-1}} \right) \\
        = & z_{n,g} \otimes 1 + 1 \otimes z_{n,g} + \sum_{\substack{k=1 \\ h \in G}}^{n-1} z_{k,h} \otimes z_{n-k,hg^{-1}}.
    \end{align*}
    On the other hand
    \[
        \hat{\Delta}^{\W}_G \circ \varpi (y_{n,g}) = \hat{\Delta}^{\W}_G(z_{n,g}) = z_{n,g} \otimes 1 + 1 \otimes z_{n,g} + \sum_{\substack{k=1 \\ h \in G}}^{n-1} z_{k,h} \otimes z_{n-k,hg^{-1}}.  
    \]
\end{proof}

\begin{lemma}
    The following diagram
    \begin{equation}
        \label{diag_link_DeltaM_Delta*}
        \begin{tikzcd}
            \KX/\KX x_0 \ar["\kappa \circ \overline{\q}^{-1}"]{rr} \ar["\hat{\Delta}^{\Mod}_{\star}"']{d} && \hat{\M}_G \ar["\hat{\Delta}^{\M}_G"]{d} \\
            (\KX/\KX x_0)^{\otimes 2} \ar["(\kappa \circ \overline{\q}^{-1})^{\otimes 2}"']{rr} && \hat{\M}_G^{\otimes 2}
        \end{tikzcd}
    \end{equation}
    commutes; where $\hat{\Delta}^{\Mod}_{\star}$ (resp. $\hat{\Delta}^{\M}_G$) is the coproduct of $\KX / \KX x_0$ (resp. $\hat{\M}_G$) given in (\ref{harmonic_coproduct_M}) (resp. (\ref{diag_DeltaW_DeltaM})).
    \label{link_DeltaM_Delta*}
\end{lemma}
\begin{proof}
    Let us consider the following cube
    \begin{equation*}
        \begin{tikzcd}[row sep={40,between origins}, column sep={40,between origins}]
            &  & \mathbf{k}\langle\langle X \rangle\rangle/\mathbf{k}\langle\langle X \rangle\rangle x_0 \arrow[ddd, "\hat{\Delta}_{\star}^{\mathrm{mod}}"] \arrow[rrrr, "\kappa \circ \overline{\q}^{-1}"] &  &  &  & \hat{\mathcal{M}}_G \arrow[ddd, "\hat{\Delta}^{\M}_G"] \\
            \mathbf{k}\langle\langle Y \rangle\rangle \arrow[rrrr, "\hspace{0.5cm} \varpi"] \arrow[ddd, "\hat{\Delta}_{\star}^{\mathrm{alg}}"'] \arrow[rru, "\pi_Y"] &  & &  & \hat{\mathcal{W}}_G \arrow[ddd, "\hat{\Delta}^{\mathcal{W}}_G"] \arrow[rru, "- \cdot 1_{\mathcal{M}}"] &  & \\
            &  &  &  &  &  &  \\
            &  & (\mathbf{k}\langle\langle X \rangle\rangle/\mathbf{k}\langle\langle X \rangle\rangle x_0)^{\otimes 2} \arrow[rrrr, "\hspace{0.5cm} (\kappa \circ \overline{\q}^{-1})^{\otimes 2}"'] &  &  &  & \hat{\mathcal{M}}_G^{\otimes 2} \\
            \mathbf{k}\langle\langle Y \rangle\rangle^{\otimes 2} \arrow[rrrr, "\varpi^{\otimes 2}"'] \arrow[rru, "\pi_Y^{\otimes 2}"] &  &  &  & \hat{\mathcal{W}}_G^{\otimes 2} \arrow[rru, "(- \cdot 1_{\mathcal{M}})^{\otimes 2}"'] &  &
        \end{tikzcd}
    \end{equation*}
    First, the left (resp. right) side commutes by definition of $\hat{\Delta}^{\Mod}_{\star}$ (resp. $\hat{\Delta}^{\M}_G$). Then, the upper side commutes thanks to Corollary \ref{iso_WG_MG} and the lower side is exactly the tensor square of the latter so is also commutative. Finally, Lemma \ref{link_DeltaW_Delta*} gives us the commutativity of the front side.
    This collection of commutativities together with the surjectivity of $\pi_Y$ implies that the back side of the cube commutes, which is exactly Diagram (\ref{diag_link_DeltaM_Delta*}).
\end{proof}

\subsubsection{Identification of stabilizer groups}
\begin{theorem}
    $\Stab(\hat{\Delta}^{\M}_G)(\K) = \Stab(\hat{\Delta}^{\Mod}_{\star})(\K)$ (as subgroups of $(\G(\KX), \circledast)$).
    \label{StabDeltaM_StabDelta*}
\end{theorem}
\begin{proof}
    Thanks to proposition \ref{link_GammaSY_GammaautM10}, the map $\kappa \circ \overline{\q}^{-1} : \KX/\KX x_0 \to \hat{\M}_G$ is a $(\G(\KX), \circledast)$-module isomorphism. So, it induces a $(\G(\KX), \circledast)$-module isomorphism $\Mor_{\K-\Mod}(\KX/\KX x_0, (\KX/\KX x_0)^{\otimes 2}) \to \Mor_{\K-\Mod}(\hat{\M}_G, \hat{\M}_G^{\otimes 2})$ which is given by
    \[
        \Delta \mapsto (\kappa \circ \overline{\q}^{-1})^{\otimes 2} \circ \Delta \circ (\kappa \circ \overline{\q}^{-1})^{-1}
    \]
    (see (\ref{act_on_Delta*}) for the definition of the $(\G(\KX), \circledast)$-module structure on the $\K$-module $\Mor_{\K-\Mod}(\KX/\KX x_0, (\KX/\KX x_0)^{\otimes 2})$ and (\ref{act_on_DeltaM}) for the $(\G(\KX), \circledast)$-module structure on the $\K$-module $\Mor_{\K-\Mod}(\hat{\M}_G, \hat{\M}_G^{\otimes 2})$.)  
    Moreover, thanks to Lemma \ref{link_DeltaM_Delta*}, the coproduct $\hat{\Delta}^{\Mod}_{\star}$ is sent to the coproduct $\hat{\Delta}^{\M}_G$ via this isomorphism. Thus, they have the same stabilizer.
\end{proof}

\subsection{The stabilizer group \texorpdfstring{$\Stab(\hat{\Delta}^{\W}_G)$}{StabWG} in Racinet's formalism} \label{sect_StabW_Stab*}
\begin{propdef}
    \label{Gamma_aut_Y}
    For $\Psi \in \G(\KX)$, we denote $^{\Gamma}\aut^Y_{\Psi}$ the $\K$-algebra automorphism of $\KY$
    \begin{equation}
        \label{Gamma_autY}
        ^{\Gamma}\aut^Y_{\Psi} := \varpi^{-1} \circ \,^{\Gamma}\aut^{\W, (1)}_{\Psi} \circ \varpi
    \end{equation}
    where $^{\Gamma}\aut^{\W, (1)}_{\Psi}$ is as in Proposition-Definition \ref{restrict_Gammaaut}, $\varpi : \KY \to \hat{\W}_G$ as in Corollary \ref{iso_WG_MG}\ref{item_diag_projections}.
    Moreover, there is a group action of $(\G(\KX), \circledast)$ on $\KY$ by $\K$-algebra automorphisms given by
    \begin{equation}
        \G(\KX) \longrightarrow \Aut_{\K-\alg}(\KY), \, \Psi \longmapsto \,^{\Gamma}\aut^Y_{\Psi} 
    \end{equation}
\end{propdef}
\begin{proof}
    For $\Psi, \Phi \in \G(\KX)$ we have
    \begin{align*}
        ^{\Gamma}\aut^Y_{\Psi \circledast \Phi} = & \varpi^{-1} \circ \,^{\Gamma}\aut^{\W, (1)}_{\Psi \circledast \Phi} \circ \varpi
        = \varpi^{-1} \circ \,^{\Gamma}\aut^{\W, (1)}_{\Psi} \circ \,^{\Gamma}\aut^{\W, (1)}_{\Phi} \circ \varpi \\
        = & \varpi^{-1} \circ \,^{\Gamma}\aut^{\W, (1)}_{\Psi} \circ \varpi \circ \varpi^{-1} \circ \,^{\Gamma}\aut^{\W, (1)}_{\Phi} \circ \varpi
        = \,^{\Gamma}\aut^Y_{\Psi} \circ \,^{\Gamma}\aut^Y_{\Phi}  
    \end{align*}
\end{proof}

We aim to give an explicit formulation of the action $^{\Gamma}\aut^{Y}$ in terms of Racinet's objects. Recall from §\ref{basic_objects} that for any $g \in G$ and any $a \in \KX$, $a x_g \in \KY$. We then have the following Lemma:

\begin{lemma}
    \label{beta_and_q}
    Let $g \in G$. For any $a \in \KX$ we have $\beta(a x_g \otimes g) = \varpi \circ \q_Y(a x_g)$.
\end{lemma}
\begin{proof}
    It is enough to show this on a basis of the $\K$-module $\KX$. For $r \in \N$, $n_1, \dots, n_{r+1} \in \N^{\ast}$ and $g_1, \dots, g_r \in G$ we have
    \begin{align*}
        & \beta(x_0^{n_1-1} x_{g_1} \cdots x_0^{n_r-1} x_{g_r} x_0^{n_{r+1}-1} x_g \otimes g) = (-1)^r e_0^{n_1-1} g_1 e_1 g_1^{-1} \cdots e_0^{n_r-1} g_r e_1 g_r^{-1} e_0^{n_{r+1}-1} g e_1 \\
        & = (-1)^r e_0^{n_1-1} g_1 e_1 \cdots e_0^{n_r-1} g_{r-1}^{-1} g_r e_1 e_0^{n_{r+1}-1} g_r^{-1}g e_1 = z_{n_1, g_1} \cdots z_{n_r, g_{r-1}^{-1} g_r} z_{n_{r+1}, g_r^{-1}g},
    \end{align*}
    where the second equality comes from a computation similar to (\ref{beta_of_basis}) and the third one from the fact that for any $i \in \{2, \dots, r\}$, $g_i^{-1} e_0 = e_0g_i^{-1}$. On the other hand,
    \begin{align*}
        & \varpi \circ \q_Y(x_0^{n_1-1} x_{g_1} \cdots x_0^{n_r-1} x_{g_r} x_0^{n_{r+1}-1} x_g) = \varpi(y_{n_1, g_1} \cdots y_{n_r, g_{r-1}^{-1}g_r} y_{n_{r+1}, g_r^{-1}g}) \\
        & = z_{n_1, g_1} \cdots z_{n_r, g_{r-1}^{-1}g_r} z_{n_{r+1}, g_r^{-1}g}.
    \end{align*}
\end{proof}

\begin{proposition}
    \label{explicit_autY}
    For $\Psi \in \G(\KX)$ and $(n, g) \in \N^{\ast} \times G$ we have
    \begin{equation}
        ^{\Gamma}\aut_{\Psi}^Y(y_{n, g}) = \q_Y\left(\Gamma_{\Psi}^{-1}(x_1) \Psi x_0^{n-1} t_g\left(\Psi^{-1} \Gamma_{\Psi}(x_1)\right) x_g\right).
    \end{equation}
\end{proposition}
\begin{proof}
    Let us start with the following computation
    \begin{align*}
        &^{\Gamma}\aut^{\W, (1)}_{\Psi}(z_{n,g}) = - \Gamma_{\Psi}^{-1}(-e_1) \beta(\Psi \otimes 1) e_0^{n-1} g \beta(\Psi^{-1} \otimes 1) e_1 \Gamma_{\Psi}(-e_1) \\
        & = - \Gamma_{\Psi}^{-1}(-e_1) \beta(\Psi \otimes 1) e_0^{n-1} g \beta(\Psi^{-1} \otimes 1) \Gamma_{\Psi}(-e_1) e_1 \\
        & = \beta((\Gamma_{\Psi}^{-1}(x_1) \otimes 1) \ast (\Psi \otimes 1) \ast (x_0^{n-1} \otimes 1) \ast (1 \otimes g) \ast (\Psi^{-1} \otimes 1) \ast (\Gamma_{\Psi}(x_1) \otimes 1) \ast (x_1 \otimes 1)) \\
        & = \beta\left(\Gamma_{\Psi}^{-1}(x_1) \Psi x_0^{n-1} t_g\left(\Psi^{-1}\Gamma_{\Psi}(x_1)\right) x_g\right) \\
        & = \varpi \circ \q_Y \left( \Gamma_{\Psi}^{-1}(x_1) \Psi x_0^{n-1} t_g\left(\Psi^{-1}\Gamma_{\Psi}(x_1)\right) x_g\right)
    \end{align*}
    where $t_g$ is the $\K$-algebra automorphism given in §\ref{basic_objects}; and the last equality comes from Lemma \ref{beta_and_q}. Thanks to this, we have for any $(n, g) \in \N^{\ast} \times G$,
    \begin{align*}
        & ^{\Gamma}\aut^Y_{\Psi}(y_{n,g}) = \varpi^{-1} \circ ^{\Gamma}\aut^{\W, (1)}_{\Psi} \circ \varpi (y_{n,g}) = \varpi^{-1} \circ \,^{\Gamma}\aut^{\W, (1)}_{\Psi}(z_{n,g}) \\
        & = \varpi^{-1} \circ \varpi \circ \q_Y \left( \Gamma_{\Psi}^{-1}(x_1) \Psi x_0^{n-1} t_g\left(\Psi^{-1}\Gamma_{\Psi}(x_1)\right) x_g\right) \\
        & = \q_Y \left( \Gamma_{\Psi}^{-1}(x_1) \Psi x_0^{n-1} t_g\left(\Psi^{-1}\Gamma_{\Psi}(x_1)\right) x_g\right).
    \end{align*}
\end{proof}

Using Proposition \ref{Gamma_aut_Y}, we define the following group action of $(\G(\KX), \circledast)$ on $\Mor_{\K-\alg}\left(\KY, \KY^{\hat{\otimes} 2}\right)$:
\begin{equation}
    \label{act_on_Deltaalg}
    \Psi \cdot D := \left({^{\Gamma}\aut^{Y}_{\Psi}}\right)^{\otimes 2} \circ D \circ \left(^{\Gamma}\aut^{Y}_{\Psi}\right)^{-1}.
\end{equation}
In particular, the stabilizer of $\hat{\Delta}^{\alg}_{\star}$ is the subgroup
\begin{equation}
    \label{Stab_Deltaalg}
    \Stab(\hat{\Delta}^{\alg}_{\star})(\K) := \left\{ \Psi \in \G(\KX) \, | \, \left( {^{\Gamma}\aut^{Y}_{\Psi}} \right)^{\otimes 2} \circ \hat{\Delta}^{\alg}_{\star} = \hat{\Delta}^{\alg}_{\star} \circ {^{\Gamma}\aut^{Y}_{\Psi}} \right\}.
\end{equation}

\begin{theorem}
    $\Stab(\hat{\Delta}^{\alg}_{\star})(\K) = \Stab(\hat{\Delta}^{\W}_{G})(\K)$ (as subgroups of $(\G(\KX), \circledast)$).
    \label{StabDeltaW_StabDelta*}
\end{theorem}
\begin{proof}
    Thanks to Proposition-Definition \ref{Gamma_aut_Y}, the map $\varpi : \KY \to \hat{\W}_G$ is an isomorphism of $(\G(\KX), \circledast)$-modules. So, it induces a $(\G(\KX), \circledast)$-module isomorphism $\Mor_{\K-\alg}(\KY, \KY^{\otimes 2}) \to \Mor_{\K-\alg}(\hat{\W}_G, \hat{\W}_G^{\otimes 2})$ which is given by
    \[
        \Delta \mapsto \varpi^{\otimes 2} \circ \Delta \circ \varpi^{-1}
    \]
    Moreover, thanks to Lemma \ref{link_DeltaW_Delta*}, the coproduct $\hat{\Delta}^{\alg}_{\star}$ is sent to the coproduct $\hat{\Delta}^{\W}_G$ via this isomorphism. Thus, they have the same stabilizer.
\end{proof}

\begin{corollary}
    \label{StabDeltaAlg_StabDeltaMod}
    $\Stab(\hat{\Delta}^{\Mod}_{\star})(\K) \subset \Stab(\hat{\Delta}^{\alg}_{\star})(\K)$ (as subgroups of $(\G(\KX), \circledast)$).
\end{corollary}
\begin{proof}
    Follows immediately from Theorem \ref{Stab_Inclusion} thanks to Theorems \ref{StabDeltaM_StabDelta*} and \ref{StabDeltaW_StabDelta*}. 
\end{proof}
	\section{Affine group scheme and Lie algebraic aspects} \label{LA_Side}
In this part, we show that the results obtained in §\ref{crossed_product} and §\ref{StabRac} fit in the framework of affine $\Q$-group schemes and make explicit the associated Lie algebraic aspects.
More precisely, we use the result of \cite{EF0} Lemma 5.1 to show that the stabilizer group functors $\Stab(\hat{\Delta}^{\M}_G)$ and $\Stab(\hat{\Delta}^{\M}_G)$ are affine $\Q$-group schemes, whose Lie algebras are stabilizer Lie algebras which we make explicit. In order to carry out this program, in §\ref{LA_actions_on_crossed_product}, we define Lie algebra actions of $(\widehat{\Lib}(X), \langle \cdot, \cdot \rangle)$ on $\hat{\V}_G^{\Q}$ by derivations and by endomorphisms. From this, we derive in §\ref{sect_stabM} endomorphism actions on $\hat{\M}_G$ that leads us to an explicit form of the Lie algebra of $\Stab(\hat{\Delta}^{\M}_G)$ that we show to be equal to the Lie algebra $\stab(\hat{\Delta}^{\Mod}_{\star})$ of (\ref{stab_Delta*}). In §\ref{sect_stabW}, we define derivation actions on $\hat{\W}_G$ that make explicit the Lie algebra $\stab(\hat{\Delta}^{\W}_G)$ of $\Stab(\hat{\Delta}^{\W}_G)$ which we show to contain $\stab(\hat{\Delta}^{\M}_G)$. In §\ref{sect_stab*}, we identify $\stab(\hat{\Delta}^{\W}_{G})$ with a Lie algebra stabilizer $\stab(\hat{\Delta}^{\alg}_{\star})$ defined in Racinet's formalism by considering the infinitesimal version of the algebra automorphism given in §\ref{sect_StabW_Stab*}. We conclude by the inclusion $\stab(\hat{\Delta}^{\Mod}_{\star}) \subset \stab(\hat{\Delta}^{\alg}_{\star})$.

\subsection{Actions of the Lie algebra \texorpdfstring{$(\widehat{\Lib}(X), \langle \cdot, \cdot \rangle)$}{LibX} on \texorpdfstring{$\hat{\V}_G^{\Q}$}{VGQ}} \label{LA_actions_on_crossed_product}

\begin{propdef}
    \label{LA_act_der_0}
    Let $\psi \in \widehat{\Lib}(X)$. There exists a unique $\Q$-algebra derivation $\der^{\V, (0)}_{\psi}$ of $\hat{\V}_G^{\Q}$ given by
    \[
        e_0 \mapsto 0, \quad e_1 \mapsto [e_1, \beta(\psi \otimes 1)], \quad g \mapsto 0, \text{ for } g \in G.
    \]
    There is a Lie algebra action of $(\widehat{\Lib}(X), \langle \cdot, \cdot \rangle)$ on $\hat{\V}_G^{\Q}$ by $\Q$-algebra derivations
        \[
            (\widehat{\Lib}(X), \langle \cdot, \cdot \rangle) \longrightarrow \mathrm{Der}_{\Q-\alg}(\hat{\V}_G^{\Q}), \psi \longmapsto \der^{\V, (0)}_{\psi}.
        \]
\end{propdef}
\begin{proof}
    One can prove that the assignment $\K \mapsto \Aut_{\K-\alg}(\hat{\V}_G)$ is a $\Q$-group scheme with Lie algebra $\mathrm{Der}_{\Q-\alg}(\hat{\V}_G^{\Q})$ and that the map $(\G(\KX), \circledast) \to \Aut_{\K-\alg}(\hat{\V}_G)$, $\Psi \mapsto \aut_{\Psi}^{\V,(0)}$ is a morphism of $\Q$-group schemes from $\K \mapsto (\G(\KX), \circledast)$ to the latter $\K \mapsto \Aut_{\K-\alg}(\hat{\V}_G)$ using Proposition \ref{group_morph_V}\ref{act_aut0}. One checks that the corresponding morphism of Lie algebras is as announced.
\end{proof}

\begin{propdef}
    \label{LA_act_der_1}
    For $\psi \in \widehat{\Lib}(X)$, we define $\der^{\V, (1)}_{\psi}$ the $\Q$-algebra derivation of $\hat{\V}_G^{\Q}$ given by
    \begin{equation}
        \der^{\V, (1)}_{\psi} = \mathrm{ad}_{\beta(\psi \otimes 1)} + \der^{\V, (0)}_{\psi}.
    \end{equation}
    There is a Lie algebra action of $(\widehat{\Lib}(X), \langle \cdot, \cdot \rangle)$ on $\hat{\V}_G^{\Q}$ by $\Q$-algebra derivations
    \[
        (\widehat{\Lib}(X), \langle \cdot, \cdot \rangle) \longrightarrow \mathrm{Der}_{\Q-\alg}(\hat{\V}_G^{\Q}), \psi \longmapsto \der^{\V, (1)}_{\psi}.
    \] 
\end{propdef}
\begin{proof}
    Same as proof of Proposition-Definition \ref{LA_act_der_0}, replacing the morphism $\Psi \mapsto \aut^{\V, (0)}_{\Psi}$ by $\Psi \mapsto \aut^{\V, (1)}_{\Psi}$ and using Proposition \ref{group_morph_V}\ref{act_aut1}.
\end{proof}

\begin{propdef}
    \label{LA_act_end_10}
    For $\psi \in \widehat{\Lib}(X)$, we define $\mathrm{end}^{\V, (10)}_{\psi}$ to be the $\Q$-linear endomorphism of $\hat{\V}_G^{\Q}$ given by
    \begin{equation}
        \label{eq_end10}
        \mathrm{end}^{\V, (10)}_{\psi} := \ell_{\beta(\psi \otimes 1)} + \der^{\V, (0)}_{\psi}.
    \end{equation}
    There is a Lie algebra action of $(\widehat{\Lib}(X), \langle \cdot, \cdot \rangle)$ on $\hat{\V}_G^{\Q}$ by $\Q$-linear endomorphisms
    \[
        (\widehat{\Lib}(X), \langle \cdot, \cdot \rangle) \longrightarrow \mathrm{End}_{\Q}(\hat{\V}_G^{\Q}), \psi \longmapsto \mathrm{end}^{\V, (10)}_{\psi}.
    \]
\end{propdef}
\begin{proof}
    Same as proof of Proposition-Definition \ref{LA_act_der_0}, introducing the $\Q$-group scheme $\K \mapsto \Aut_{\K-\Mod}(\hat{\V}_G)$, whose Lie algebra is $\End_{\Q}(\hat{\V}_G^{\Q})$, and viewing $\Psi \mapsto \aut^{\V,(10)}_{\Psi}$ as a $\Q$-group scheme morphism from $\K \mapsto (\G(\KX), \circledast))$ to $\K \mapsto \Aut_{\K-\Mod}(\hat{\V}_G)$ thanks to Proposition \ref{act_aut10}. 
\end{proof}

\subsection{The stabilizer Lie algebra \texorpdfstring{$\stab(\hat{\Delta}^{\M}_G)$}{stabDMG}} \label{sect_stabM}
\begin{propdef}
    \label{end_M_10}
    For $\psi \in \widehat{\Lib}(X)$, there is a unique $\Q$-linear endomorphism $\mathrm{end}^{\M, (10)}_{\psi}$ of $\hat{\M}_G^{\Q}$ such that the following diagram commutes
    \begin{equation}
        \begin{tikzcd}
            \hat{\V}_G^{\Q} \ar[rrr, "\mathrm{end}^{\V, (10)}_{\psi}"] \ar[d, "- \cdot 1_{\M}"'] &&& \hat{\V}_G^{\Q} \ar[d, "- \cdot 1_{\M}"] \\
            \hat{\M}_G^{\Q} \ar[rrr, "\mathrm{end}^{\M, (10)}_{\psi}"'] &&& \hat{\M}_G^{\Q}
        \end{tikzcd}
    \end{equation}
    There is a Lie algebra action of $(\widehat{\Lib}(X), \langle \cdot, \cdot \rangle)$ on $\hat{\M}_G$ by $\Q$-linear endomorphisms
    \[
        (\widehat{\Lib}(X), \langle \cdot, \cdot \rangle) \longrightarrow \End_{\Q}(\hat{\M}_G^{\Q}), \psi \longmapsto \mathrm{end}^{\M, (10)}_{\psi}.
    \]
\end{propdef}
\begin{proof}
    The commutative diagram is given by applying Proposition-Definition \ref{aut_M10} for $\K=\Q[\epsilon]/(\epsilon^2)$ and $\psi \in \ker\Big(\G(\KX) \to \G(\Q\langle\langle X \rangle\rangle)\Big)$.
    For the second statement, one first checks that the assignment $\K \mapsto \Aut_{\K-\Mod}(\hat{\M}_G)$ is an affine $\Q$-group scheme whose Lie algebra is $\End_{\Q}(\hat{\M}_G^{\Q})$. Then, using Corollary \ref{group_morph_aut_M_10}, one deduces that the map $\Psi \mapsto \aut^{\M, (10)}_{\Psi}$ is a $\Q$-group scheme morphism from $\left(\K \mapsto (\G(\KX), \circledast)\right)$ to $\left(\K \mapsto \Aut_{\K-\Mod}(\hat{\M}_G)\right)$. One finally proves that $\mathrm{end}^{\M, (10)}$ is its corresponding Lie algebra morphism. 
\end{proof}

To $\psi \in \widehat{\Lib}(X)$, one associates $\gamma_{\psi} \in \Q[[x]]$ (see (\ref{gamma_function})). Then $\gamma_{\psi}(-e_1)$ is an element of $\hat{\V}_G^{\Q}$.

\begin{propdef}
    \label{LA_act_gamma_M}
    For $\psi \in \widehat{\Lib}(X)$, we define the following $\Q$-linear endomorphism of $\hat{\M}_G^{\Q}$:
    \begin{equation}
        \label{gamma_end_M_10}
        ^{\gamma}\mathrm{end}^{\M, (10)}_{\psi} := \ell_{-\gamma_{\psi}(-e_1)} + \mathrm{end}^{\M, (10)}_{\psi}. 
    \end{equation} 
    There is a Lie algebra action of $(\widehat{\Lib}(X), \langle \cdot, \cdot \rangle)$ on $\hat{\M}_G^{\Q}$ by $\Q$-linear endomorphisms
    \begin{equation}
        (\widehat{\Lib}(X), \langle \cdot, \cdot \rangle) \longrightarrow \mathrm{End}_{\Q}(\hat{\M}_G^{\Q}), \psi \longmapsto \,^{\gamma}\mathrm{end}^{\M, (10)}_{\psi}
    \end{equation}
\end{propdef}
\begin{proof}
    The maps $\Psi \mapsto \aut^{\M, (10)}_{\Psi}$ and $\Psi \mapsto \,^{\Gamma}\aut_{\Psi}^{\M, (10)}$ are $\Q$-group scheme morphisms from $\left(\K \mapsto (\G(\KX), \circledast)\right)$ to $\left(\K \mapsto \Aut_{\K-\Mod}(\hat{\M}_G)\right)$. The $\Q$-Lie algebra morphism associated to the former $\Q$-group scheme morphism is $\psi \mapsto \mathrm{end}_{\psi}^{\M, (10)}$ by the proof of Proposition-Definition \ref{end_M_10}.
    The Lie algebra morphism associated to the latter $\Q$-group scheme morphism takes $\psi \in \widehat{\Lib}(X)$ to the right hand side of (\ref{gamma_end_M_10}) in view of (\ref{Def_Gamma_aut_M_10}), therefore is given by $\psi \mapsto \,^{\gamma}\mathrm{end}_{\psi}^{\M, (10)}$. It follows that the latter map is a Lie algebra morphism.
\end{proof}

Thanks to this result, we are able to provide a Lie algebra action of $(\widehat{\Lib}(X), \langle \cdot, \cdot \rangle)$ on the space $\Mor_{\Q}\left(\hat{\M}_G^{\Q}, \left(\hat{\M}_G^{\Q}\right)^{\hat{\otimes} 2}\right)$ via
\begin{equation}
    \label{LA_act_on_DeltaM}
    \psi \cdot D^{\M} := \left({^{\gamma}\mathrm{end}^{\M, (10)}_{\psi}} \otimes \mathrm{id} + \mathrm{id} \otimes \, ^{\gamma}\mathrm{end}^{\M, (10)}_{\psi} \right) \circ D^{\M} - D^{\M} \circ {^{\gamma}\mathrm{end}^{\M, (10)}_{\psi}}.
\end{equation}
In particular, the stabilizer of $\hat{\Delta}^{\M}_{G}$ is the Lie subalgebra
\begin{equation}
    \stab(\hat{\Delta}^{\M}_{G}) :=
    \left\{\begin{array}{l}
        \psi \in \widehat{\Lib}(X) \, | \\
        \left({^{\gamma}\mathrm{end}^{\M, (10)}_{\psi}} \otimes \mathrm{id} + \mathrm{id} \otimes \, ^{\gamma}\mathrm{end}^{\M, (10)}_{\psi} \right) \circ D^{\M} = D^{\M} \circ {^{\gamma}\mathrm{end}^{\M, (10)}_{\psi}}
    \end{array}\right\}
\end{equation}

\noindent For a commutative $\Q$-algebra $\K$, recall the group $\Stab(\hat{\Delta}^{\M}_G)(\K)$ in (\ref{Stab_DeltaM}). One then has
\begin{proposition}
    \label{stab_is_LA_Stab} The assignment $\Stab(\hat{\Delta}^{\M}_G) : \K \mapsto \Stab(\hat{\Delta}^{\M}_G)(\K)$ is an affine $\Q$-group scheme and $\mathbf{Lie}(\Stab(\hat{\Delta}^{\M}_G)) = \stab(\hat{\Delta}^{\M}_G)$.
\end{proposition}
\begin{proof}
    The first statement is obtained by applying \cite{EF0}, Lemma 5.1 where $v=\hat{\Delta}^{\M}_G$ and the second one comes from the fact that the $(\widehat{\Lib}(X), \langle \cdot, \cdot \rangle)$-action on $\Mor_{\Q}\left(\hat{\M}_G^{\Q}, \left(\hat{\M}_G^{\Q}\right)^{\hat{\otimes} 2}\right)$ given in (\ref{LA_act_on_DeltaM}) is the infinitesimal version of the group action of $(\G(\KX), \circledast)$ on $\Mor_{\K-\Mod}\left(\hat{\M}_G, \left(\hat{\M}_G\right)^{\hat{\otimes} 2}\right)$ given in (\ref{act_on_DeltaM}), for any $\Q$-algebra $\K$.
\end{proof}

\begin{corollary}
    \label{stabDeltaM_stabDeltamod} $\stab(\hat{\Delta}^{\M}_G) = \stab(\hat{\Delta}^{\Mod}_{\star})$ (as Lie subalgebras of $(\widehat{\Lib}(X), \langle \cdot, \cdot \rangle)$).
\end{corollary}
\begin{proof}
    It follows from Theorem \ref{StabDeltaM_StabDelta*} thanks to Propositions \ref{stab_is_LA_Stab} and \ref{LA_of_Rac_groups}\ref{LA_stab_Stab}.
\end{proof}

\subsection{The stabilizer Lie algebra \texorpdfstring{$\stab(\hat{\Delta}^{\W}_G)$}{stabDWG}} \label{sect_stabW}
\begin{propdef}
    \label{LA_act_gamma_der_V}
    For $\psi \in \widehat{\Lib}(X)$, we define the $\Q$-algebra derivation of $\hat{\V}_G^{\Q}$:
    \begin{equation}
        ^{\gamma}\der^{\V, (1)}_{\psi} := \ad_{-\gamma_{\psi}(-e_1)} + \der^{\V, (1)}_{\psi}. 
    \end{equation}
    There is a Lie algebra action of $(\widehat{\Lib}(X), \langle \cdot, \cdot \rangle)$ on $\hat{\V}_G^{\Q}$ by $\Q$-algebra derivations
    \begin{equation}
        \label{gamma_der_V_1}
        (\widehat{\Lib}(X), \langle \cdot, \cdot \rangle) \longrightarrow \mathrm{Der}_{\Q-\alg}(\hat{\V}_G^{\Q}), \, \psi \longmapsto \,^{\gamma}\der^{\V, (1)}_{\psi}.
    \end{equation} 
\end{propdef}
\begin{proof}
    The maps $\Psi \mapsto \aut^{\V, (1)}_{\Psi}$ and $\Psi \mapsto \,^{\Gamma}\aut_{\Psi}^{\V, (1)}$ are $\Q$-group scheme morphisms from $\left(\K \mapsto (\G(\KX), \circledast)\right)$ to $\left(\K \mapsto \Aut_{\K-\alg}(\hat{\V}_G)\right)$. The $\Q$-Lie algebra morphism associated to the former $\Q$-group scheme morphism is $\psi \mapsto \der_{\psi}^{\V, (1)}$ by the proof of Proposition-Definition \ref{LA_act_der_1}.
    The Lie algebra morphism associated to the latter $\Q$-group scheme morphism takes $\psi \in \widehat{\Lib}(X)$ to the right hand side of (\ref{gamma_der_V_1}) in view of (\ref{Def_Gamma_aut_V_1}), therefore is given by $\psi \mapsto \,^{\gamma}\der_{\psi}^{\V, (1)}$. It follows that the latter map is a Lie algebra morphism.
\end{proof}

\begin{propdef}
    \label{LA_act_gamma_der_W}
    For $\psi \in \widehat{\Lib}(X)$, the derivation $^{\gamma}\der^{\V, (1)}_{\psi}$ restricts to a derivation of the subalgebra $\hat{\W}_G^{\Q}$ denoted $^{\gamma}\der^{\W, (1)}_{\psi}$. Moreover, there is a Lie algebra action of $(\widehat{\Lib}(X), \langle \cdot, \cdot \rangle)$ on $\hat{\W}_G^{\Q}$ by $\Q$-algebra derivations
    \begin{equation}
        (\widehat{\Lib}(X), \langle \cdot, \cdot \rangle) \longrightarrow \mathrm{Der}_{\Q-\alg}(\hat{\W}_G^{\Q}), \, \psi \longmapsto ^{\gamma}\der^{\W, (1)}_{\psi}
    \end{equation}
\end{propdef}
\begin{proof}
    One can prove that the assignment $\K \mapsto \Aut_{\K-\alg}(\hat{\W}_G)$ is a $\Q$-group scheme with Lie algebra $\mathrm{Der}_{\Q-\alg}(\hat{\W}_G^{\Q})$. The map $\Psi \mapsto \,^{\Gamma}\aut_{\Psi}^{\V, (1)}$ is a $\Q$-group scheme morphism from $\left(\K \mapsto (\G(\KX), \circledast)\right)$ to $\left(\K \mapsto \Aut_{\K-\alg}(\hat{\V}_G)\right)$ and, by the proof of Proposition-Definition \ref{LA_act_gamma_der_V}, its associated $\Q$-Lie algebra is $\psi \mapsto \, ^{\gamma}\der_{\psi}^{\V, (1)}$. Thanks to Proposition \ref{Gamma_aut_VW_1}\ref{Gamma_aut_W_1}, we obtain the following commutative diagram
    \[
        \begin{tikzcd}
            \hat{\W}_G \ar["^{\Gamma}\aut_{\Psi}^{\W, (1)}"]{rr} \ar[hook]{d} && \hat{\W}_G \ar[hook]{d} \\
            \hat{\V}_G \ar["^{\Gamma}\aut_{\Psi}^{\V, (1)}"']{rr} && \hat{\V}_G
        \end{tikzcd} 
    \]
    where $\Psi \in \G(\KX)$ with $\K$ a commutative $\Q$-algebra.
    Using this diagram for $\K=\Q[\epsilon]/(\epsilon^2)$ and $\psi \in \ker\Big(\G(\KX) \to \G(\Q\langle\langle X \rangle\rangle)\Big)$, one obtains that the derivation $^{\gamma}\der^{\V, (1)}_{\psi}$ restricts to a derivation on $\hat{\W}^{\Q}_G$ associated to the automorphism $^{\Gamma}\aut_{\Psi}^{\W, (1)}$, which we denoted $^{\gamma}\der_{\psi}^{\W, (1)}$.
    Moreover, the diagram states that $\Q$-group scheme morphism provided by $\Psi \mapsto \,^{\Gamma}\aut_{\Psi}^{\V, (1)}$ defines a $\Q$-group scheme morphism $\Psi \mapsto \,^{\Gamma}\aut_{\Psi}^{\W, (1)}$ from $\left(\K \mapsto (\G(\KX), \circledast)\right)$ to $\left(\K \mapsto \Aut_{\K-\alg}(\hat{\W}_G)\right)$. Therefore, the map $\psi \mapsto \, ^{\gamma}\der_{\psi}^{\W, (1)}$ from $(\widehat{\Lib}(X), \langle \cdot, \cdot \rangle)$ to $\mathrm{Der}_{\Q-\alg}(\hat{\W}_G^{\Q})$ which is the infintesimal version of the latter $\Q$-group scheme morphism is a $\Q$-Lie algebra morphism.
\end{proof}

Using Proposition-Definition \ref{LA_act_gamma_der_W}, one can define the following Lie algebra action of $(\widehat{\Lib}(X), \langle \cdot, \cdot \rangle)$ on the space $\Mor_{\Q}\left(\hat{\W}_G^{\Q}, \left(\hat{\W}_G^{\Q}\right)^{\hat{\otimes} 2}\right)$:
\begin{equation}
    \label{LA_act_on_DeltaW}
    \psi \cdot D^{\W} := \left({^{\gamma}\der^{\W, (1)}_{\psi}} \otimes \mathrm{id} + \mathrm{id} \otimes \, ^{\gamma}\der^{\W, (1)}_{\psi} \right) \circ D^{\W} - D^{\W} \circ {^{\gamma}\der^{\W, (1)}_{\psi}}.
\end{equation}
In particular, the stabilizer of $\hat{\Delta}^{\W}_{G}$ is the Lie subalgebra
\begin{equation}
    \mathfrak{stab}(\hat{\Delta}^{\W}_{G}) :=
    \left\{ \begin{array}{l}
    \psi \in \widehat{\Lib}(X) \, | \\
    \left({^{\gamma}\der^{\W, (1)}_{\psi}} \otimes \mathrm{id} + \mathrm{id} \otimes \, ^{\gamma}\der^{\W, (1)}_{\psi} \right) \circ D^{\W} = D^{\W} \circ {^{\gamma}\der^{\W, (1)}_{\psi}}
    \end{array}\right\}.
\end{equation}

\noindent For a commutative $\Q$-algebra $\K$, recall the group $\Stab(\hat{\Delta}^{\W}_G)(\K)$ in (\ref{Stab_DeltaW}). One then has
\begin{proposition}
    \label{stabW_is_LA_StabW} The assignment $\Stab(\hat{\Delta}^{\W}_G) : \K \mapsto \Stab(\hat{\Delta}^{\W}_G)(\K)$ is an affine $\Q$-group scheme and $\mathbf{Lie}(\Stab(\hat{\Delta}^{\W}_G)) = \stab(\hat{\Delta}^{\W}_G)$.
\end{proposition}
\begin{proof}
    The first statement is obtained by applying \cite{EF0}, Lemma 5.1 where $v=\hat{\Delta}^{\W}_G$ and the second one comes from the fact that the Lie algebra action of $(\widehat{\Lib}(X), \langle \cdot, \cdot \rangle)$ on $\Mor_{\Q}\left(\hat{\W}_G^{\Q}, \left(\hat{\W}_G^{\Q}\right)^{\hat{\otimes} 2}\right)$ given in (\ref{LA_act_on_DeltaW}) is the infinitesimal version of the group action of $(\G(\KX), \circledast)$ on $\Mor_{\K-\Mod}\left(\hat{\W}_G, \left(\hat{\W}_G\right)^{\hat{\otimes} 2}\right)$ given in (\ref{act_on_DeltaW}), for any $\Q$-algebra $\K$.
\end{proof}

\begin{corollary}
    \label{stab_inclusion} $\mathfrak{stab}(\hat{\Delta}^{\M}_{G}) \subset \mathfrak{stab}(\hat{\Delta}^{\W}_{G})$ (as Lie subalgebras of $(\widehat{\Lib}(X), \langle \cdot, \cdot \rangle$).
\end{corollary}
\begin{proof}
    It follows from Theorem \ref{Stab_Inclusion} thanks to Propositions \ref{stabW_is_LA_StabW} and \ref{stab_is_LA_Stab}.
\end{proof}

\subsection{The stabilizer Lie algebra \texorpdfstring{$\stab(\hat{\Delta}^{\W}_G)$}{stabWG} in Racinet's formalism} \label{sect_stab*}
\begin{propdef}
    \label{gamma_der_Y}
    For $\psi \in \widehat{\Lib}(X)$, we denote $^{\gamma}d^Y_{\psi}$ the derivation of $\Q\langle\langle Y \rangle\rangle$ given by
    \begin{equation}
        \label{gamma_dY}
        ^{\gamma}d^Y_{\psi} := \varpi^{-1} \circ \,^{\gamma}\der^{\W, (1)}_{\psi} \circ \varpi
    \end{equation}
    where $^{\gamma}\der^{\W, (1)}_{\psi}$ is as in Proposition-Definition \ref{LA_act_gamma_der_W} and $\varpi : \Q\langle\langle Y \rangle\rangle \to \hat{\W}_G^{\Q}$ is the $\Q$-algebra isomorphism of Corollary \ref{iso_WG_MG}\ref{item_diag_projections}.
    There is a Lie algebra action of $(\widehat{\Lib}(X), \langle \cdot, \cdot \rangle)$ on $\Q\langle\langle Y \rangle\rangle$ by derivations given by
    \begin{equation}
        \widehat{\Lib}(X) \longrightarrow \mathrm{Der}_{\Q-\alg}(\Q\langle\langle Y \rangle\rangle), \, \psi \longmapsto \,^{\gamma}d^Y_{\psi} 
    \end{equation}
\end{propdef}
\begin{proof}
    One can prove that the assignment $\K \mapsto \Aut_{\K-\alg}(\KY)$ is a $\Q$-group scheme with Lie algebra $\mathrm{Der}_{\Q-\alg}(\QY)$. Thanks to Proposition-Definition \ref{Gamma_aut_Y}, the map $(\G(\KX), \circledast) \to \Aut_{\K-\alg}(\KY)$, $\Psi \mapsto \, ^{\Gamma}\aut_{\Psi}^{Y}$ is a morphism of $\Q$-group schemes from $\K \mapsto (\G(\KX), \circledast)$ to the latter $\K \mapsto \Aut_{\K-\alg}(\KY)$.
    It is related to the morphism of $\Q$-group schemes $\Psi \mapsto \,^{\Gamma} \aut^{\W, (1)}_{\Psi}$ of Proposition-Definition \ref{restrict_Gammaaut} by (\ref{Gamma_autY}). It follows that the corresponding $\Q$-Lie algebra morphism takes $\psi \in \LibX$ to the right hand side of (\ref{gamma_dY}). The statement then follows from (\ref{gamma_dY}).
\end{proof}

For any $\psi \in \widehat{\Lib}(X)$, the derivation $^{\gamma}d^Y_{\psi}$ can be expressed in the formalism of \cite{Rac} as follows
\begin{proposition}
    For $\psi \in \widehat{\Lib}(X)$ and $(n, g) \in \N^{\ast} \times G$ we have
    \begin{equation}
        \label{gamma_d_psi_ng}
        ^{\gamma}d_{\psi}^Y(y_{n, g}) = \q_Y\Big(\big(\psi x_0^{n-1} - x_0^{n-1} t_g(\psi)\big) x_g\Big) + \q_Y\Big(\big(x_0^{n-1} \gamma_{\psi}(x_g) -  \gamma_{\psi}(x_1)x_0^{n-1}\big) x_g\Big).
    \end{equation}
\end{proposition}
\begin{proof}
    The infinitesimal version of the identity in Proposition \ref{explicit_autY} is given by
    \[
        ^{\gamma}d_{\psi}^Y(y_{n, g}) = \q_Y\Big(\big((-\gamma_{\psi}(x_1) + \psi) x_0^{n-1} + x_0^{n-1} t_g(\gamma_{\psi}(x_1) - \psi)\big) x_g\Big).
    \]
    Identity \ref{gamma_d_psi_ng} then follows.
\end{proof}

From Proposition \ref{gamma_der_Y}, we define a Lie algebra action of $(\widehat{\Lib}(X), \langle \cdot, \cdot \rangle)$ on the space $\Mor_{\Q}\left(\Q\langle\langle Y \rangle\rangle, \Q\langle\langle Y \rangle\rangle^{\hat{\otimes} 2}\right)$ by
\begin{equation}
    \label{LA_act_on_Deltaalg}
    \psi \cdot D := \left({^{\gamma}d^{Y}_{\psi}} \otimes \mathrm{id} + \mathrm{id} \otimes {^{\gamma}d^{Y}_{\psi}}\right) \circ D - D \circ \,^{\gamma}d^{Y}_{\psi}.
\end{equation}
In particular, the stabilizer of $\hat{\Delta}^{\alg}_{\star}$ is the Lie subalgebra
\begin{equation}
    \stab(\hat{\Delta}^{\alg}_{\star}) := \left\{ \psi \in \widehat{\Lib}(X) \, | \, \left({^{\gamma}d^{Y}_{\psi}} \otimes \mathrm{id} + \mathrm{id} \otimes {^{\gamma}d^{Y}_{\psi}}\right) \circ \hat{\Delta}^{\alg}_{\star} = \hat{\Delta}^{\alg}_{\star} \circ \,^{\gamma}d^{Y}_{\psi} \right\}.
\end{equation}

\noindent For a commutative $\Q$-algebra $\K$, recall the group $\Stab(\hat{\Delta}^{\alg}_{\star})(\K)$ in (\ref{Stab_Deltaalg}). One then has
\begin{proposition}
    \label{stabalg_is_LA_Stabalg} The assignment $\Stab(\hat{\Delta}^{\alg}_{\star}) : \K \mapsto \Stab(\hat{\Delta}^{\alg}_{\star})(\K)$ is an affine $\Q$-group scheme and $\mathbf{Lie}(\Stab(\hat{\Delta}^{\alg}_{\star})) = \stab(\hat{\Delta}^{\alg}_{\star})$.
\end{proposition}
\begin{proof}
    The first statement is a consequence of \cite{EF0}, Lemma 5.1 where $v= \hat{\Delta}^{\alg}_{\star}$ and the second one comes from the fact that the $(\widehat{\Lib}(X), \langle \cdot, \cdot \rangle)$-action on $\Mor_{\Q}\left(\Q\langle\langle Y \rangle\rangle, \Q\langle\langle Y \rangle\rangle^{\hat{\otimes} 2}\right)$ given in (\ref{LA_act_on_Deltaalg}) is the infinitesimal version of the group action of $(\G(\KX), \circledast)$ on $\Mor_{\K-\alg}\left(\K\langle\langle Y \rangle\rangle, \K\langle\langle Y \rangle\rangle^{\hat{\otimes} 2}\right)$ given in (\ref{act_on_Deltaalg}), for any $\Q$-algebra $\K$.
\end{proof}

\begin{corollary}
    \label{stabDeltaW_stabDelta*} $\stab(\hat{\Delta}^{\alg}_{\star}) = \stab(\hat{\Delta}^{\W}_{G})$ (as Lie subalgebras of $(\widehat{\Lib}(X), \langle \cdot, \cdot \rangle$).
\end{corollary}
\begin{proof}
    It follows from Theorem \ref{StabDeltaW_StabDelta*} thanks to Propositions \ref{stabalg_is_LA_Stabalg} and \ref{stabW_is_LA_StabW}.
\end{proof}

\noindent Finally, in Racinet's formalism, this translates to:
\begin{corollary}
    $\stab(\hat{\Delta}^{\Mod}_{\star}) \subset \stab(\hat{\Delta}^{\alg}_{\star})$ (as Lie subalgebras of $(\widehat{\Lib}(X), \langle \cdot, \cdot \rangle$).
\end{corollary}
\begin{proof}
    It immediately follows from Corollary \ref{stab_inclusion} thanks to Corollaries \ref{stabDeltaM_stabDeltamod} and \ref{stabDeltaW_stabDelta*}.  
\end{proof}
	\nocite{Enr20} \nocite{EF2} \nocite{EF3} \nocite{Enr08} \nocite{Fur11} \nocite{Fur12} \nocite{Kas} \nocite{DeGa}
	\bibliographystyle{abstract}
	\bibliography{BibliPaperOne}
\end{document}